\documentclass[a4paper,12pt,draft]{amsart}
\usepackage[margin=30mm]{geometry}
\usepackage{amssymb, amsmath, amsthm, amscd}

\usepackage[T1]{fontenc}
\usepackage{eucal,mathrsfs,dsfont}
\usepackage{color}
\usepackage[normalem]{ulem}

\definecolor{mno}{rgb}{0.5,0.1,0.5}

\newcommand{\Rd}{{\mathds R^d}}
\newcommand{\Pp}{\mathds P}
\newcommand{\Ee}{\mathds E}
\newcommand{\E}{\mathscr E}
\newcommand{\cE}{\mathscr E}

\newcommand{\I}{\mathds 1}

\def\<{\langle}
\def\>{\rangle}

\newtheorem{theorem}{Theorem}[section]
\newtheorem{lemma}[theorem]{Lemma}
\newtheorem{prop}[theorem]{Proposition}
\newtheorem{corollary}[theorem]{Corollary}

\theoremstyle{definition}

\newtheorem{remark}[theorem]{Remark}

\makeatletter 

\@addtoreset{equation}{section}
\makeatother 

\theoremstyle{remark}

\numberwithin{equation}{section}
\numberwithin{theorem}{section}

\def \RR {\mathds R}
\def \NN {\mathds{N}}

\def \lm {\lambda}

\def \tp{\tilde{p}}

\def \RdO {\mathds{R}^d_0}

\begin{document}
\allowdisplaybreaks

\title[Fractional Schr\"{o}dinger operators with negative hardy potential] {\bfseries  Heat kernel estimates of fractional  Schr\"{o}dinger operators with negative hardy potential}
\author{Tomasz Jakubowski\quad and \quad Jian Wang}
\address{\emph{T.\ Jakubowski}: Faculty of Pure and Applied Mathematics, Wroc{\l}aw University of Science and Technology, Wyb. Wyspia\'{n}skiego 27, 50-370 Wroc{\l}aw, Poland.}
\email{tomasz.jakubowski@pwr.edu.pl}
\thanks{Tomasz Jakubowski is partially supported by the NCN grant 2015/18/E/ST1/00239. Jian Wang is partially 
supported by the NNSFC (Nos.\
11522106 and 11831014), the Fok Ying Tung Education Foundation (No.\ 151002),
the Program for Probability and Statistics: Theory and Application
(No. IRTL1704), and the Program for IRTSTFJ}
\address{\emph{J.\ Wang}:
   College of Mathematics and Informatics \& Fujian Key Laboratory of Mathematical Analysis and Applications (FJKLMAA), Fujian Normal University, 350007 Fuzhou, P.R. China.}
    \email{jianwang@fjnu.edu.cn}

\begin{abstract} We obtain two-sided estimates for the heat kernel (or the fundamental
function) associated with the following fractional Schr\"{o}dinger operator with
negative Hardy potential
$$\Delta^{\alpha/2} -\lambda |x|^{-\alpha}$$ on $\RR^d$, where $\alpha\in(0,d\wedge
2)$ and $\lambda>0$. The proof is purely analytical but elementary.
In particular, for upper bounds of heat kernel we use the Chapman-Kolmogorov equation and adopt self-improving argument.
\end{abstract}

\maketitle

\section{Introduction}
Let $d\in \NN_+:=\{1,2,\cdots\}$ and $\alpha\in(0,d\wedge 2)$. We consider the following
Schr\"{o}dinger operator
\begin{equation}\label{e:op1}
\mathcal{L} := \Delta^{\alpha/2} +q
\end{equation}
on $\RR^d$, where $\Delta^{\alpha/2}:=-(-\Delta)^{\alpha/2}$ and
$q(x) = \kappa|x|^{-\alpha}$ with
\begin{equation}\label{e:constant}\kappa=\kappa_\delta:= \frac{2^\alpha\Gamma(\frac{\alpha-\delta}{2}) \Gamma(\frac{d+\delta}{2})}
{\Gamma(\frac{-\delta}{2})\Gamma(\frac{d+\delta-\alpha}{2})}
\end{equation} for any $\delta \in(0,\alpha).$ Here,
$\Gamma(-\delta/2)=\int_0^\infty r^{-1-\delta/2}(e^{-r}-1)\,dr<0$
for $\delta\in (0,\alpha)$, and $\Gamma(z)=\int_0^\infty
r^{z-1}e^{-r}\,dr$ for all $z>0$.  We note that $\kappa_\delta<0$,
and so $q(x) <0$ on $\RR^d$; we also note that the radial function
$|x|^{-\alpha}$ comes from the Hardy inequality for fractional
Laplacian $\Delta^{\alpha/2}$ (see \cite{2016-KB-BD-PK-pa, FLS} and
the references therein for more details). Thus, the operator
$\mathcal{L}$ given by \eqref{e:op1} is the fractional
Schr\"{o}dinger operator with negative Hardy potential. Denote by
$\tp(t,x,y)$ the heat kernel associated with the operator
$\mathcal{L}=\Delta^{\alpha/2} +q$; see Subsection \ref{e:sub2-1}
below for more details. Our main result is as follows.
\begin{theorem}\label{th:main}For any $\delta\in (0,\alpha)$, the Schr\"{o}dinger operator $\mathcal{L}$ given by \eqref{e:op1} has the heat kernel $\tp(t,x,y)$, which is jointly continuous
on $(0,\infty)\times \RR^d\times \RR^d$, and satisfies two-sided
estimates as follows
\begin{align}\label{EQ:main}
    \tp(t,x,y) \approx \left(1\wedge \frac{|x|}{t^{1/\alpha}}\right)^\delta \left(1\wedge \frac{|y|}{t^{1/\alpha}}\right)^\delta \left(t^{-d/\alpha}\wedge\frac{t}{|x-y|^{d+\alpha}}\right),\quad x,y\in \RR^d,\;  t>0.
\end{align}
\end{theorem}
We note that the last expression in \eqref{EQ:main} may be  replaced
by the heat kernel $p(t,x,y)$ of $\Delta^{\alpha/2}$ (see Subsection
2.1 and \eqref{eq:pest}). As pointed out before Lemma \ref{L:2.3}
below, the function $\delta\mapsto \kappa_\delta$ is strictly
decreasing on $(0,\alpha)$ with $\lim_{\delta\to0}\kappa_\delta=0$
and $\lim_{\delta\to \alpha}\kappa_\delta=-\infty.$ Hence, Theorem
\ref{th:main} essentially gives us two-sided estimates and the joint
continuity of heat kernel associated with the operator
$\Delta^{\alpha/2} -\lambda |x|^{-\alpha}$ for all $\lambda>0$. It
is well known that the fractional Laplacian $\Delta^{\alpha/2}$ is
the infinitesimal generator of the rotationally symmetric
$\alpha$-stable process, which now has been attracted a lot of
interests in the field of probability and potential theory (see
\cite{BBK} and references therein). Recently there are also a few
works concerning on gradient perturbations and Schr\"{o}dinger
perturbations of fractional Laplacian (see e.g.
\cite{2008-KB-WH-TJ-sm, BJ, BJ2012, 2013-KB-TJ-SS-jee, CKS3, CKS2,
2009-TJ-pa, KK,  2013-TK-pa, Song1, Take}).  In particular,
according to \cite[Theorem 3.4]{Song1}, when the potential belongs
to the so-called Kato class, heat kernel estimates for
Schr\"{o}dinger perturbations of fractional Laplacian are comparable
with these for fractional Laplacian (at least for any fixed finite
time). Note that $q(x)=-\lambda|x|^{-\alpha}$ does not belong to the
Kato class. As shown in Theorem \ref{th:main}, the heat kernel
$\tp(t,x,y)$ associated with the Schr\"{o}dinger operator
$\mathcal{L}$ given by \eqref{e:op1} exhibits behaviour which is
different from that of the case that $q(x)=-\lambda|x|^{-\gamma}$
with $\gamma\in (0,\alpha),$ which is in the Kato class. The study
of heat kernel estimates for Schr\"odinger-type perturbations by the
Hardy potential of fractional Laplacian is much more delicate.

In the classical case $\alpha=2$, the Schr\"odinger-type
perturbations by the Hardy potential were considered for the first
time by Baras and Goldstein \cite{MR742415}. They proved the
existence of nontrivial nonnegative solutions of the classical heat
equation $\partial_t=\Delta+\kappa|x|^{-2}$ in $\Rd$ for $0\le
\kappa\le (d-2)^2/4$, and nonexistence of such solutions, that is
explosion, for bigger constants $\kappa$. Sharp upper and lower
bounds for the heat kernel of the Schr\"odinger operator
$\Delta+\kappa|x|^{-2}$ were obtained by Liskevich and Sobol
\cite[p.~365 and Examples 3.8, 4.5 and 4.10]{2003-VL-ZS-pa}  for
$0<\kappa< (d-2)^2/4$. Milman and Semenov proved the upper and lower
bounds for $\kappa\le (d-2)^2/4$, see \cite[Theorem
1]{2004-PM-YS-jfa} and \cite{MR2180080}. In this paper, they also
allowed $\kappa <0$ and obtained the sharp upper and lower bounds
for the perturbed kernel (see \cite[Theorem 2 and Corollary
4]{2004-PM-YS-jfa}). See \cite{IKO} and the references therein for
the recent works of this topic.

For $\alpha \in (0, d \land 2)$ the Schr\"odinger operator
$\mathcal{L}$ with $\kappa \ge 0$ attains recently more  and more
interest. In \cite{AMPP1,AMPP2} for $ \kappa > \kappa^* :=
\frac{2^{\alpha} \Gamma((d+\alpha)/4)^2 }{\Gamma((d-\alpha)/4)^{2}}$
the phenomenon of instantaneous blow up of heat kernel was proven.
In \cite{BA}, the author gives the upper bound for the heat kernel
of $\mathcal{L}$ with the Dirichlet conditions on bounded open
subsets of $\RR^d$. In the recent paper \cite{2017-KB-TG-TJ-DP}, the
following sharp estimates for the heat kernel $\tilde{p}(t,x,y)$ of
$\mathcal{L}$ were obtained. For $0 \le \kappa \le \kappa^*$, there
is a unique constant $\delta \in [0,(d-\alpha)/2]$ such that for all
$t>0$ and $x,y \in \RR^d_0:=\RR^d\backslash\{0\}$,
\begin{equation}
\label{EQ:positEst} \tilde{p}(t,x,y)\approx
\left(1+t^{\delta/\alpha}|x|^{-\delta}
\right)\left(1+t^{\delta/\alpha}|y|^{-\delta}  \right) p(t,x,y).
\end{equation}
Note that since the singularity  of the function $\Rd\ni
x\mapsto\kappa |x|^{-\alpha}$ at the origin is critical,
$\tilde{p}(t,x,y)$ is not comparable with the unperturbed kernel
$p(t,x,y)$. Like in Theorem \ref{th:main}, the choice of $\kappa$
influences the growth rate or the decay rate of the heat kernel at
the origin. This rate is represented by the function
$|x|^{-\delta}$, where $\delta$ is connected with $\kappa$ via the
formula $\kappa=\frac{2^\alpha\Gamma(\frac{\alpha+\delta}{2})
\Gamma(\frac{d-\delta}{2})}{\Gamma(\frac{\delta}{2})\Gamma(\frac{d-\delta-\alpha}{2})}$
(compared with \eqref{e:constant}).
\begin{remark}Theorem \ref{th:main} and  \cite[Theorem 1.1]{2017-KB-TG-TJ-DP} (the main result of \cite{2017-KB-TG-TJ-DP}) can be stated together as follows.
\emph{ For any $-\infty < \kappa \le \kappa^*=\frac{2^{\alpha}
\Gamma((d+\alpha)/4)^2 }{\Gamma((d-\alpha)/4)^{2}}$, the heat kernel
$\tp(t,x,y)$ corresponding to the Schr\"odinger operator
$\Delta^{\alpha/2}+\kappa |x|^{-\alpha}$ satisfies
\begin{align*}
\tilde{p}(t,x,y)\approx \left(1+\frac{t^{1/\alpha}}{|x|^{\alpha}}  \right)^\delta \left(1+\frac{t^{1/\alpha}}{|y|^{\alpha}}  \right)^\delta p(t,x,y), \qquad t>0,\; x,y \in \RdO,
\end{align*}
where $\delta \in (-\alpha,\frac{d-\alpha}{2}]$ is uniquely determined by
\begin{align*}
\kappa= \frac{2^\alpha\Gamma(\frac{\alpha+\delta}{2}) \Gamma(\frac{d-\delta}{2})}
{\Gamma(\frac{\delta}{2})\Gamma(\frac{d-\delta-\alpha}{2})}.
\end{align*}}
\end{remark}
 In this setting, Theorem \ref{th:main} may be treated as both a fractional counterpart of the result
  obtained in \cite{2004-PM-YS-jfa} and the extension of \eqref{EQ:positEst} to negative values of $\kappa$.
  Here, we would like to stretch out one difference between the cases $\alpha=2$ and $\alpha<2$ for $\kappa<0$.
  The general form of the estimate in both cases is similar, i.e., the perturbed kernel $\tp(t,x,y)$ is comparable
  with the unperturbed kernel $p(t,x,y)$ multiplied by some weighted functions. However, in \cite[Theorem 2 and Corollary 4]{2004-PM-YS-jfa}, for $\alpha=2$,
  the exponent of the weighted function is equal to $\delta=  \frac{ \sqrt{(d-2)^2 - 4\kappa} - (d-2)}{2} $ and converges to infinity as $\kappa \to -\infty$. In our case $\alpha<2$, as it was mentioned below the statement of Theorem \ref{th:main}, $\delta \to \alpha$ for $\kappa \to -\infty$.
Since $q(x)=\kappa |x|^{-\alpha}$ is negative and does not belong to
any Kato class on $\RR^d$, the construction and proofs of the
estimates of $\tp(t,x,y)$ are very delicate. In particular, we
cannot use the perturbation series (at least for large values of
$-\kappa$) to construct $\tp(t,x,y)$ as used in
\cite{2016-KB-BD-PK-pa,2017-KB-TG-TJ-DP,2008-KB-WH-TJ-sm}. That is
why we will consider the Dirichlet  fractional Laplacian  operator
$\Delta^{\alpha/2}$ on $\RdO = \RR^d \setminus \{0\}$ and via the
Feyman-Kac formula, we construct $\tp(t,x,y)$ on $(0,\infty)\times
\RdO \times \RdO$. Hence, the operator $\mathcal{L}$ with negative
values of $\kappa$ also enjoys some probabilistic meaning. Roughly
speaking, it is connected with a symmetric $\alpha$-stable process
with the killing rate  $e^{-\kappa|x|^{-\alpha}}$, which strongly
affects the behaviour of $\tp(t,x,y)$ for $x$  and $y$ near $0$. It
turns out that due to the strong singularity of $q(x)$ at $0$, the
heat kernel (or the transition density function) $\tp(t,x,y)$ is
equal to $0$ when $x=0$ or $y=0$. In consequence, the kernel
$\tp(t,x,y)$ defined on $(0,\infty) \times \RdO \times \RdO$ may be
continuously extended to $(0,\infty) \times \RR^d \times \RR^d$.

We note that Theorem \ref{th:main} was proved independently in a very recent paper \cite{CKSV}.
 In  the proofs, the authors use generally probabilistic tools. In our paper we propose a different method. Although
 the perturbed kernel $\tp(t,x,y)$ is defined by the Feyman-Kac formula,
 in the proofs we apply only analytical tools. For upper bounds, we generally use the Chapman-Kolmogorov equation and the method of \lq\lq \emph{self-improving estimates}'' (see the proofs of Proposition \ref{lem:mass_est} and Theorem \ref{thm:UB}, see also the proof of \cite[Theorem 1.1]{2016-TJ-GS-natmaa}). Roughly speaking, to show the inequality $f(x) \le C F(x)$, we first show that $f(x) \le g_1(x) + c_1 F(x)$, where $g_1(x)$ is in some sense small. Next, by plugging this estimate to the proper functional inequality on $f$, we get the improved estimate of the form $f(x) \le g_n(x) + c_n F(x)$, where $g_n(x) \to 0$ as $n \to \infty$ and $\sup_{n \in \NN_+}c_n <\infty$. By passing with $n$ to infinity we obtain the desired estimate. To obtain lower bounds we use the generally well known estimate from  Lemma \ref{lem:tpexp} and upper bound estimates. Although the estimate from Lemma \ref{lem:tpexp} is generally well known,  we couldn't find the proper reference with the assumptions on the potential satisfied by $q(x)$. We note that the setting of \cite{CKSV} is more general than the present paper. From the other side, we give more details about the kernel $\tilde{p}(t,x,y)$, see e.g. Theorem \ref{thm:0}. We also note that in our paper we show the straightforward dependence between the exponent $\delta$ and the potential $q(x)$, while in \cite[Theorem 3.9]{CKSV} this dependence, given by double integral, is much more complicated.

The paper is organized as follows. In Section 2, we construct
$\tp(t,x,y)$ and prove some basic properties of this kernel. In
Section 3,  we give the proof of Theorem \ref{th:main}. First, we
prove upper bounds in Theorem \ref{thm:UB}. Next, we show lower
bounds in Theorem \ref{Thm:LB} and joint continuity (Theorem
\ref{thm:cont}). We end this section with short discussion on
Dirichlet forms associated with the Schr\"{o}dinger operator
$\mathcal{L}$ given by \eqref{e:op1}. Finally, in the Appendix, we
present the proof Lemma \ref{lem:tpexp}.

Throughout the paper, we write $f\approx g$ for $f,g\ge0$, if there
is a constant $c\ge1$ such that $c^{-1}f\le g\le c f$ on their
common domain. The constants $c, C, c_i$, whose exact values are
unimportant, are changed in each statement and proof. Let $B(x,r)$
be the open ball with center $x\in \RR^d$ and radius $r>0$. As usual
we write $a  \land b := \min(a,b)$ and $a \vee b := \max(a,b)$.

\section{Preliminary estimates}
\subsection{Fractional Laplacian and rotationally symmetric $\alpha$-stable L\'evy process}
Let
\begin{equation}\label{e:levymea}
\nu(z)
=\frac{ \alpha 2^{\alpha-1}\Gamma\big((d+\alpha)/2\big)}{\pi^{d/2}\Gamma(1-\alpha/2)}|z|^{-d-\alpha}\,,\quad z\in \Rd.
\end{equation}
For (smooth and compactly supported) test function $\varphi\in C^\infty_c(\Rd)$, we define the fractional Laplacian  by
\begin{equation*}
  \Delta^{\alpha/2}\varphi(x)=-(-\Delta)^{\alpha/2}\varphi(x):=
  \lim_{\varepsilon \downarrow 0}\int_{\{|z|>\varepsilon\}}
  \left[\varphi(x+z)-\varphi(x)\right]\nu(z)\,dz\,,
  \quad
  x\in \Rd\,.
\end{equation*}
In terms of the Fourier transform (see \cite[Section~1.1.2]{FLS}),
$\widehat{\Delta^{\alpha/2}\varphi}(\xi)=-|\xi|^{\alpha}\hat{\varphi}(\xi)$.
Denote by $p(t,x,y)$ the heat kernel (or the fundamental function) of  $\Delta^{\alpha/2}$ (or equivalently, the transition density
function of a (rotationally) symmetric  $\alpha$-stable L\'evy
process $(X_t)_{t\ge0}$). It is well known that $p(t,x,y)$ is
symmetric in the sense that $p(t,x,y)=p(t,y,x)$ for any $t>0$ and
$x,y\in \RR^d$, and enjoys the following scaling property
$$p(t,x,y)=t^{-d/\alpha}p(1,t^{-1/\alpha} x, t^{-1/\alpha}y),\quad t>0,\; x,y\in \RR^d.$$
Moreover,
\begin{align}\label{eq:pest}
    p(t,x,y) \approx t^{-d/\alpha} \land \frac{t}{|x-y|^{d+\alpha}},\quad t>0,\;x,y\in \RR^d.
\end{align}
We also note that $p(t,x,y)$ is a function of $t$ and $x-y$, so
sometimes we also write it as $p(t,x-y)$, i.e. $p(t,x,y) =
p(t,x-y)$.
See \cite{BBK} for more details.

\subsection{Fractional Laplacian Schr\"{o}dinger operator and Feynman-Kac formula}\label{e:sub2-1}
In this part, we apply some results from \cite[Chapter 2]{DC} to the
operator $\mathcal{L}=\Delta^{\alpha/2} +q$ given by \eqref{e:op1}, where
$q(x)=\kappa_\delta |x|^{-\alpha}<0$.  Let
$\RR_0^d:=\RR^d\backslash\{0\}$. We first recall \cite[Chapter
2, Definition 2.1]{DC}.  A nonnegative Borel measurable function $V$
on $\RR^d_0$ is said to belong to \emph{the Kato class} $\mathcal
K_\alpha$, if
$$\lim_{t\to0}\sup_{x\in \RR^d_0} \int_0^t\int_{\RR^d_0} p(s,x,y)V(y)\,dy\,ds=0.$$
A nonnegative Borel measurable function $V$ on $\RR^d_0$ is said to
belong to \emph{the local Kato class} $\mathcal K_{\alpha, {\rm
loc}}$, if $V\I_D\in \mathcal K_{\alpha}$ for all compact subsets $D$ of
$\RR^d_0$. A Borel measurable function $V$ on $\RR^d_0$ is said to
belong to \emph{the Kato-Feller class}, if its positive part
$V_+:=\max\{V,0\}\in \mathcal K_\alpha$ and its negative part
$V_-:=\max\{-V,0\}\in \mathcal K_{\alpha, {\rm loc}}.$ (Different
from \cite{DC}, in the present setting we start from the nonpositive
definite operator $\Delta^{\alpha/2} +q$, and so we make the
corresponding changes in the definition of the Kato-Feller class.)
It is easily seen from \cite[Lemma 2.3]{2017-KB-TG-TJ-DP} that
$-q\notin \mathcal K_\alpha$, but always we have  $-q\in \mathcal
K_{\alpha, {\rm loc}}$. In particular, $q$ belongs to the
Kato-Feller class.

In the following, we will restrict ourselves on the killed
subprocess of the symmetric $\alpha$-stable L\'evy process
$(X_t)_{t\ge0}$ upon exiting $\RR_0^d$ (or hitting the origin), i.e.,
\begin{equation*}
X^{\RR_0^d}_t:=\begin{cases}
X_t,\quad\text{if}\ t<\tau_{\RR_0^d},\\
\,0,\,\,\quad \text{if}\ t\ge \tau_{\RR_0^d},
\end{cases}
\end{equation*}
where $\tau_{\RR_0^d}:=\inf \{t \ge 0: X_t \notin \RR_0^d\}=\inf\{t\ge0:X_t=0\}$. By the strong Markov property of the process $(X_t)_{t\ge0}$, it is easy to see that the process $(X_t^{\RR_0^d})_{t\ge0}$ has a transition density (or Dirichlet heat kernel) $p^{\RR_0^d}(t,x,y)$, which enjoys the following relation with $p(t,x,y)$:
\begin{equation*}
\begin{split} p^{\RR_0^d}(t,x,y)&=p(t,x,y)-\Ee^x\big[p(t-\tau_{\RR_0^d},X_{\tau_{\RR_0^d}},y)\I_{\{t\ge \tau_{\RR_0^d}\}}\big],\quad x, y \in {\RR_0^d};\\
p^{\RR_0^d}(t,x,y)&=0,\quad x=0 \text{ or } y =0.
\end{split}
\end{equation*} Since the process $(X_t)_{t\ge0}$ is transient due to $\alpha<d$, $\Pp^x(\tau_{\RR_0^d} < \infty)=0$ for all $x\in \RR_0^d$, and consequently
\begin{equation}\label{e:hk1}p^{\RR_0^d}(t,x,y)=p(t,x,y),\quad t>0,\;x,y\in \RR_0^d.\end{equation}

It is well known that for every $t>0$ the function $p(t,\cdot,\cdot)$ is
continuous on ${\RR^d}\times {\RR^d}$, and $p(t,x,y)$ satisfies the following Chapman-Kolmogorov equation
\begin{align}\label{eq:CKp}
	p(t+s,x,y)=\int_{\RR^d} p(t,x,z)p(s,z,y)\,dz,\quad t,s>0,\; x,y\in \RR^d.
\end{align}

Regard $\mathcal{L}=\Delta^{\alpha/2} +q$ as the operator defined on
$C_\infty(\RR_0^d)$; that is, we consider a negative perturbation of
the fractional Laplacian on $\RR^d_0$ (with the Dirichlet boundary condition at $\{0\}$). Therefore, according to
\cite[Theorem 2.5]{DC}, the operator $\mathcal{L}=\Delta^{\alpha/2} +q$ can
generate a strongly continuous and positivity preserving semigroup
$(\tilde P_t)_{t\ge0}$ on $C_\infty(\RR_0^d)$, which is given by
$$
\tilde P_t f(x)=\int \tp(t,x,y)f(y)\,dy,\quad f\in C_\infty (\RR^d_0),\; x\in \RR_0^d,
$$
where the kernel $\tp(t,x,y)$ satisfies the Chapman-Kolmogorov equation too, i.e.,
\begin{align}\label{eq:CKpt}
	\tp(t+s,x,y)=\int \tp(t,x,z)\tp(s,z,y)\,dz,\quad t,s>0,\; x,y\in \RR_0^d.
\end{align}
Additionally, for $t>0$, we put $\tp(t,x,y) =0$, whenever $x=0$ or
$y=0$. Moreover, $(\tilde P_t)_{t\ge0}$ also acts as  a strongly
continuous semigroup in $L^p(\RR_0^d;dx)$ for all $1\le p<\infty$,
and, in $L^2(\RR_0^d;dx)$ the semigroup $(\tilde P_t)_{t\ge0}$ is
self-adjoint. Meanwhile, $(\tilde P_t)_{t\ge0}$ is given via the
Feynman-Kac formula:
\begin{equation}\label{e:hk2}
\begin{split}
\tilde P_t f(x)=&\Ee^x \left( f\big(X^{\RR^d_0}_t\big)e^{\int_0^t q\big(X^{\RR^d_0}_s\big)\,ds}\right)\\
=&\Ee^x \left( f(X_t)e^{\int_0^t
q(X_s)\,ds}\I_{\{t<\tau_{\RR_0^d}\}}\right),\quad f\in
C_\infty(\RR_0^d),\; t>0,\; x\in \RR^d_0.\end{split}
\end{equation}
As mentioned above, $\Pp^x(\tau_{\RR_0^d} < \infty)=0$ for all
$x\in \RR_0^d$, thus we have
\begin{equation}\label{e:hk3}\tilde P_t f(x)=\Ee^x \left( f(X_t)e^{\int_0^t q(X_s)\,ds}\right),\quad f\in C_\infty(\RR_0^d),\; t>0,\;x\in \RR_0^d.\end{equation}

Due to $q(x)<0$ on $\RR^d_0$ again, it follows from \eqref{e:hk1} and
\eqref{e:hk2} that
$$\tp(t,x,y) \le p(t,x,y),\quad t>0,\;  x,y\in \RR^d.
$$
Since $(\tilde P_t)_{t\ge0}$ is self-adjoint in $L^2(\RR_0^d;dx)$,
$$\tp(t,x,y)=\tp(t,y,x),\quad t>0,\; x,y\in \RR^d.$$ According
to \cite[Propositions 5.2 and 5.3]{DC} and their proofs,
$\tp(t,x,y)$ will satisfy the following Duhamel's formula:
\begin{equation}\label{e:for}
\begin{split}\tp(t,x,y) &= p(t,x,y) + \int_0^t \int_{\RR^d} p(t-s,x,z)q(z)\tp(s,z,y)\, dz\,ds\\
&=p(t,x,y)+\int_0^t \int_{\RR^d} \tp(t-s,x,z)q(z)p(s,z,y)\, dz\,ds
\end{split}
\end{equation} for all $t>0$ and $x,y\in \RR_0^d$.

Next, we show that $\tp(t,x,y)$ enjoys the same scaling property as $p(t,x,y)$.

\begin{lemma}\label{L:sca} For any $t>0$ and $x,y\in \RR^d$,
$$\tp (t,x,y)=t^{-d/\alpha} \tp(1,xt^{-1/\alpha},yt^{-1/\alpha}).$$ \end{lemma}
\begin{proof}
We only consider the case that $x,y\in \RR_0^d$; otherwise, the
statement holds trivially. Recall that for the symmetric
$\alpha$-stable process $(X_t)_{t\ge0}$, the processes
$(X_{ut})_{t\ge0}$ and $(t^{1/\alpha}X_u)_{t\ge0}$ enjoy the same
law for any fixed $u>0$. For fixed $t>0$, set $\hat
X_u=X_{ut}$  for $u\ge0$. Then, by \eqref{e:hk3}, for any $f\in
C_\infty(\RR^d_0)$, $t>0$ and $x\in \RR_0^d$,
\begin{align*}\tilde P_t f(x)=&\Ee^x \left( f(X_t)e^{\int_0^t q(X_s)\,ds}\right)=\Ee^x \left( f(\hat X_1)e^{\int_0^t q(\hat X_{s/t})\,ds}\right)=
\Ee^x \left( f(\hat X_1)e^{t\int_0^1 q(\hat X_{u})\,du}\right)\\
=&\Ee^{t^{-1/\alpha}x} \left( f(t^{1/\alpha} X_1)e^{t\int_0^1 q(
t^{1/\alpha}X_{u})\,du}\right)=\Ee^{t^{-1/\alpha}x} \left(
f(t^{1/\alpha} X_1)e^{\int_0^1 q(X_{u})\,du}\right),\end{align*}
where in the last equality we used the fact that $q(x)=\kappa_\delta
|x|^{-\alpha}.$ Hence, the desired assertion follows from the
equality above.
\end{proof}

\subsection{Integral analysis for fractional Laplacian Schr\"{o}dinger operator}
\begin{lemma} Let $\beta \in (0,2)$. Then,
\begin{align}\label{e:pos}
|x|^\beta = 2^{d+\beta} \pi^{d/2} \frac{\Gamma\left(\frac{d+\beta}{2}\right)}{\Gamma\left(\frac{d+\beta}{\alpha}\right) \left|\Gamma\left(\frac{-\beta}{2}\right)\right|} \int_0^\infty [p(t,0) - p(t,x)]t^{\frac{d-\alpha+\beta}{\alpha}} dt\,
\end{align} where $\Gamma(-\beta/2)=\int_0^\infty t^{-1-\beta/2}(e^{-t}-1)\,dt$.
\end{lemma}
\begin{proof}
We follow the method used in the proof of \cite[Proposition 5]{2016-KB-BD-PK-pa}. First, let $\eta_t(s)$ be the density function of the distribution of the $\alpha/2$-stable subordinator at time $t$. Let $g_t(x) = (4\pi t)^{-d/2} e^{-|x|^2/(4t)}$. Then,
\begin{align*}
p(t,x) = \int_0^\infty g_s(x) \eta_t(s)\, ds.
\end{align*}

By \cite[(24)]{2016-KB-BD-PK-pa}, for $\gamma < d/2-1$,
\begin{align*}
\int_0^\infty g_s(x) s^{\gamma} \,ds = 4^{-\gamma-1}\pi^{-d/2} \Gamma(d/2-\gamma-1) |x|^{2\gamma-d+2}.
\end{align*}
Then, by integrating by parts, for $d/2-1 < \gamma < d/2$, we get
\begin{align*}
\int_0^\infty (g_s(0)-g_s(x)) s^{\gamma} \,ds &= (4\pi)^{-d/2} \int_0^\infty (1-e^{-\frac{|x|^2}{4s}}) s^{\gamma-d/2} \,ds \\
&= \frac{(4\pi)^{-d/2}}{\gamma+1-d/2}  \int_0^\infty \frac{|x|^2}{4s^2}  e^{-\frac{|x|^2}{4s}} s^{\gamma+1-d/2} \,ds \\
&= \frac{|x|^2}{4(\gamma+1-d/2)}  \int_0^\infty g_s(x) s^{\gamma-1} \,ds \\
&= 4^{-\gamma-1}\pi^{-d/2} \frac{\Gamma(d/2-\gamma)}{\gamma+1-d/2} |x|^{2\gamma-d+2}.
\end{align*}
Note that, for any $\gamma>-1$,
\begin{align*}
\int_0^\infty t^\gamma \eta_t(s)\, dt = \frac{\Gamma(\gamma+1)}{\Gamma(\frac{\alpha(\gamma+1)}{2})} s^{\frac{\alpha(\gamma+1)}{2}-1},
\end{align*} see \cite[(23)]{2016-KB-BD-PK-pa}. (Note that the condition that $\gamma<d/\alpha -1$ is not required in the proof of \cite[(23)]{2016-KB-BD-PK-pa}.)
We further obtain
\begin{align*}
\int_0^\infty [p(t,0) - p(t,x)] t^{\frac{d-\alpha+\beta}{\alpha}} \,dt
&=\int_0^\infty \int_0^\infty[g_s(0) - g_s(x)] \eta_t(s) t^{\frac{d-\alpha+\beta}{\alpha}}\, dt\, ds  \\
&=\frac{\Gamma(\frac{d+\beta}{\alpha})}{\Gamma(\frac{d+\beta}{2})}\int_0^\infty[g_s(0) - g_s(x)] s^{\frac{d+\beta-2}{2}}\, ds \\
&= 2^{-d-\beta}\pi^{-d/2}\frac{\Gamma(\frac{d+\beta}{\alpha})}{\Gamma(\frac{d+\beta}{2})} \frac{\Gamma(\frac{2-\beta}{2})}{\frac{\beta}{2}} |x|^\beta.
\end{align*}
Since $\Gamma(\frac{2-\beta}{2}) = -\frac{\beta}{2} \Gamma(-\frac{\beta}{2})$, we get the assertion of the lemma.
\end{proof}

We recall from \cite[(25)]{2016-KB-BD-PK-pa} that for any $\beta \in
(0,d)$,
\begin{align}\label{e:neg}
|x|^{-\beta} = 2^{d-\beta} \pi^{d/2} \frac{\Gamma\left(\frac{d-\beta}{2}\right)}{\Gamma\left(\frac{d-\beta}{\alpha}\right) \Gamma\left(\frac{\beta}{2}\right)} \int_0^\infty p(t,x)t^{\frac{d-\alpha-\beta}{\alpha}} \,dt, \quad x\in \RR^d.
\end{align}
Thus, \eqref{e:pos} may be treated as an extension of the formula \eqref{e:neg} to negative $\beta$. Note that in the proof of \eqref{e:pos} we have to use a compensated kernel $p(t,0)$ to ensure convergence of the integral involved.

Now, let $\beta \in (0,\alpha)$. By \eqref{e:neg},
\begin{align*}
|x|^{\beta-\alpha} = 2^{d+\beta-\alpha} \pi^{d/2} \frac{\Gamma\left(\frac{d+\beta-\alpha}{2}\right)}{\Gamma\left(\frac{d+\beta-\alpha}{\alpha}\right) \Gamma\left(\frac{\alpha-\beta}{2}\right)} \int_0^\infty p(t,x)t^{\frac{d-2\alpha+\beta}{\alpha}}\,dt.
\end{align*}
On the other hand, let $f(r) = c r^{(d-\alpha+\beta)/\alpha}$ with
\begin{equation}\label{e:constant00}c=2^{d+\beta} \pi^{d/2} \frac{\Gamma\left(\frac{d+\beta}{2}\right)}{\Gamma\left(\frac{d+\beta}{\alpha}\right) \left|\Gamma\left(\frac{-\beta}{2}\right)\right|}.\end{equation} Then, according to \eqref{e:pos},
\begin{equation}\label{e:01}
    |x|^{\beta} = \int_0^\infty [p(t,0) -p(t,x)] f(t) \,dt.
\end{equation} Combining two equations above together, we will find that
\begin{align*}
    -\kappa_\beta |x|^{-\alpha} = \frac{\int_0^\infty p(r,x) f'(r) \,dr}{\int_0^\infty [p(r,0) -p(r,x)] f(r) \,dr},
\end{align*}
where
$$
    \kappa_\beta = \frac{2^\alpha\Gamma(\frac{\alpha-\beta}{2}) \Gamma(\frac{d+\beta}{2})}{\Gamma(\frac{-\beta}{2})\Gamma(\frac{d+\beta-\alpha}{2})}.
$$
In particular,
\begin{align}\label{e:02}
    -\kappa_\beta |x|^{\beta-\alpha} = \int_0^\infty p(r,x) f'(r) \,dr.
\end{align}
We note that $\kappa_\beta <0$ for any $\beta\in(0,\alpha)$ and $\lim_{\beta \to \alpha} \kappa_\beta = -\infty$. For convenience, let $\kappa_0=0.$
Moreover,
write
$$\kappa_\beta=-\frac{2^\alpha\Gamma(\frac{\alpha-\beta}{2}) \Gamma(\frac{d+\beta}{2})\frac{\beta}{2}}{\Gamma(\frac{2-\beta}{2})\Gamma(\frac{d+\beta-\alpha}{2})},$$ and let  $$r(t)=\frac{\Gamma(\frac{\alpha}{2}-t)\Gamma(\frac{d}{2}+t)}{\Gamma(\frac{d-\alpha}{2}+t)\Gamma(1-t)},\quad 0<t<\alpha/2.$$ Then, using the formula
$$ \frac{\Gamma'(x)}{\Gamma(x)}=-\gamma-\sum_{k=0}^\infty\left(\frac{1}{x+k}-\frac{1}{1+k}\right),\quad x>0$$
with the Euler-Mascheroni constant $\gamma$ (see \cite[(1.2.13)]{AAR}) and following the argument in the end of the proof for \cite[Proposition 5]{2016-KB-BD-PK-pa}, one can check that $r(t)$ is strictly increasing on $(0,\alpha/2)$, and so $\beta\mapsto \kappa_\beta$ is strictly decreasing on $(0,\alpha)$.

\begin{lemma}\label{L:2.3} For $\beta \in (0,\alpha)$, $t>0$ and $x\in\RR^d$, we have
\begin{equation}\label{e:03}
\int_{\RR^d} p(t,x,y)|y|^{\beta}\, dy  = |x|^{\beta} - \kappa_\beta\int_0^t \int_{\RR^d} p(s,x,y) |y|^{\beta-\alpha} \,dy \,ds
\end{equation}
\end{lemma}
\begin{proof}Let $f(r) = c r^{(d-\alpha+\beta)/\alpha}$ with the constant $c$ given by \eqref{e:constant00}.
By \eqref{e:02} and \eqref{e:01}, for any $t>0$ and $x\in \RR^d$,
\begin{align*}
-\kappa_\beta\int_0^t \int_{\RR^d} p(s,x,y) |y|^{\beta-\alpha} \,dy \,ds  &=  \int_0^t \int_0^\infty p(s+r,x) f'(r)\, dr\, ds\\
&= -\int_0^t \int_0^\infty \frac{\partial}{\partial s} p(s+r,x) f(r)\, dr\, ds\\
&=  \int_0^\infty [ p(r,x)-p(t+r,x)] f(r) \,dr \\
&= \int_0^\infty [ p(r,x)- p(r,0) + p(r,0)\! -\!p(t+r,x)] f(r)\, dr \\
&= -|x|^{\beta} \!+\!\! \int_0^\infty\! \!\int_{\RR^d} \!p(t,x,y)[p(r,0)\!-\!p(r,y)]f(r) \,dy\, dr\\
&= -|x|^{\beta} +  \int_{\RR^d} p(t,x,y) |y|^{\beta}\,dy,
\end{align*} where in the second equality we used the fact that $$ \lim_{r \to \infty}p(s+r,x) f(r)\le c_1\lim_{r\to\infty} (s+r)^{-d/\alpha} f(r)=0.$$ This completes the proof.
\end{proof}
Set $h_\beta(x)=|x|^\beta$. Letting $t\to0$ in \eqref{e:03},
informally it holds that \begin{equation}\label{e:eee}(\Delta^{\alpha/2}+\kappa_\beta|x|^{-\alpha}) h_\beta(x)=0\end{equation} for all $x\in \RR^d$. That is, the function $h_\beta$ is harmonic with respect to the operator $\Delta^{\alpha/2}+\kappa_\beta|x|^{-\alpha}.$

\emph{From now, we will fix $\delta\in (0,\alpha)$, and write $\kappa_\delta$ as $\kappa$ for simplicity.}
The following theorem is an analog of \cite[Theorem 3.1]{2017-KB-TG-TJ-DP}. Since there is no problem with convergence of the integrals involved, the proof is much simpler than that of \cite[Theorem 3.1]{2017-KB-TG-TJ-DP}.

\begin{theorem}\label{thm:0}
For $\beta \in (0, \alpha)$, $t>0$ and $x\in \RR^d$, we have
\begin{align}\label{eq:0}
    \int_{\RR^d} \tp(t,x,y)  |y|^{\beta}\, dy = |x|^{\beta} + (\kappa - \kappa_\beta) \int_0^t \int_{\RR^d} \tp(s,x,y) |y|^{\beta-\alpha}\,dy\,ds.
\end{align} In particular, for any $t>0$ and $x\in \RR^d$,
\begin{align}\label{eq:00}
\int_{\RR^d} \tp(t,x,y)  |y|^{\delta}\, dy = |x|^{\delta}.
\end{align}
\end{theorem}
\begin{proof} When $x=0$, both sides of \eqref{eq:0} and \eqref{eq:00} are equal to zero, since $\tp(t,0,y)=0$ for all $t>0$ and $y\in \RR^d$. Below, we consider the case that $x\in \RR_0^d$.
By \eqref{e:for} and \eqref{e:03}, for any $t>0$ and $x\in \RR_0^d$,
\begin{align*}
& -\kappa_\beta\int_0^t \int_{\RR^d} \tp(s,x,y) |y|^{\beta-\alpha}\, dy \,ds \\
&= -\kappa_\beta\int_0^t \int_{\RR^d} p(s,x,y)|y|^{\beta-\alpha}\, dy\, ds \\
&\quad- \kappa_\beta\int_0^t \int_{\RR^d} \int_u^t \int_{\RR^d} \tp(u,x,z) q(z) p(s-u,z,y) |y|^{\beta-\alpha}\,dz\, du\,  dy\, ds
 \\
&= -|x|^{\beta} + \int_{\RR^d} p(t,x,y) |y|^{\beta} \,dy\\
&
\quad-\int_0^t \int_{\RR^d} \tp(u,x,z) q(z) \left(|z|^\beta -\int_{\RR^d} p(t-u,z,y) |y|^{
\beta}  dy \right) \,dz\, du.
\end{align*}
Hence, according to \eqref{e:for} again, for any $t>0$ and $x\in \RR_0^d$,
\begin{align*}
-(\kappa_\beta - \kappa)\int_0^t \int_{\RR^d} \tp(s,x,y) |y|^{\beta-\alpha}\, dy\, ds &= -|x|^{\beta} + \int_{\RR^d} p(t,x,y) |y|^{\beta} \,dy \\
&\quad + \int_{\RR^d} (\tp(t,x,y)-  p(t,x,y)) |y|^{\beta}\,  dy   \\
&= -|x|^{\beta} + \int_{\RR^d} \tp(t,x,y) |y|^{\beta} \,dy,
\end{align*} which proves \eqref{eq:0}.
Now, \eqref{eq:00} follows by taking $\beta=\delta$. The proof is complete.
\end{proof}

Note that \eqref{e:eee} implies that
$$[\Delta^{\alpha/2}+\kappa |x|^{-\alpha}]h_\beta(x)=(\kappa-\kappa_\beta)|x|^{\beta-\alpha},\quad x\in \RR^d.$$ For this it is easy to obtain \eqref{eq:0} via the Feynman-Kac semigroup associated with $\Delta^{\alpha/2}+\kappa |x|^{-\alpha}.$

Although the following lemma is  not used in  the proofs,  we state it as one of the results. From this lemma we see that the right-hand side of \eqref{eq:logx} behaves near $0$ as $-\log|x|$.
\begin{lemma}
For any $t>0$ and $x\in \RR^d$, it holds that
\begin{align}\label{eq:logx}
    C\int_{\RR^d} \tp(t,x,y)  |y|^{\delta} (\ln|y| - \ln|x|)\, dy =  \int_0^t \int_{\RR^d} \tp(s,x,y) |y|^{\delta-\alpha}\,dy\,ds,
\end{align} where $$C:=\lim_{\beta \to \delta} \frac{\delta-\beta}{\kappa_\beta-\kappa}>0.$$
\end{lemma}
\begin{proof}
By \eqref{eq:0}, \eqref{eq:00} and the dominated convergence theorem,
\begin{align*}
\int_0^t \int_{\RR^d} \tp(s,x,y) |y|^{\delta-\alpha}\,dy\,ds &= \lim_{\beta \to \delta} \frac{1}{\kappa_\beta-\kappa} \int_{\RR^d} \tp(t,x,y)  |y|^{\delta} (|x|^{\beta-\delta} - |y|^{\beta-\delta})\, dy \\
&= C \int_{\RR^d} \tp(t,x,y)  |y|^{\delta} (\ln|y| - \ln|x|) \,dy,
\end{align*} proving the desired assertion.
\end{proof}
\section{Two-sided estimates and joint continuity of $\tp(t,x,y)$}
\subsection{Upper bounds of $\tp(1,x,y)$}
For any $t>0$ and $x\in \RR^d$, define
$$
H(t,x) = \int_{\RR^d} \tp(t,x,y)\,dy.
$$
Note that, by Lemma \ref{L:sca}, for all $t>0$ and $x\in \RR^d$, we have
\begin{align}\label {eq:H_1}
H(t,x) = \int_{\RR^d} t^{-d/\alpha}\tp(1,t^{-1/\alpha}x,t^{-1/\alpha}y)\, dy = H(1,t^{-1/\alpha}x).
\end{align}
On the other hand, by the fact $0 \le \tp(t,x,y)\le p(t,x,y)$ for any $t>0$ and $x,y\in \RR^d$, it also holds that
\begin{align}\label{eq:H_2}
0 \le H(t,x) \le \int_{\RR^d} p(t,x,y)\, dy = 1,\quad t>0,\; x\in \RR^d.
\end{align}

\begin{prop}\label{lem:mass_est}
There is a constant $C>0$ such that for all $x\in \RR^d$,
$$
H(1,x) \le C(1 \land |x|^\delta).
$$
\end{prop}
\begin{proof}
By the Chapman-Kolmogorov equation \eqref{eq:CKpt} (which holds true for all $x,y\in \RR^d$) and \eqref{eq:H_1}, for any
$x,y\in \RR^d   $,
\begin{equation}\label{eq:tp13}\begin{split}
\tp(1,x,y) &= \int_{\RR^d} \int_{\RR^d} \tp(1/3,x,z) \tp(1/3,z,w) \tp(1/3,w,y)\, dw \,dz \\
&\le \int_{\RR^d} \int_{\RR^d} \tp(1/3,x,z)\cdot c \cdot \tp(1/3,w,y)\, dw\, dz \\
&= c H(1/3,x) H(1/3,y) = c H(1,3^{1/\alpha}x) H(1,3^{1/\alpha}y),
\end{split}\end{equation}
where the constant $c$ comes from the estimate $\tp(1/3,x,y) \le p(1/3,x,y) \le c$.

Denote by $|B(0,r)|$ the Lebesgue measure of $B(0,r)$.  Fix $r>0$
small enough such that $\eta := c|B(0,r)|<3^{-\delta/\alpha}$.
According to  \eqref{eq:0}, \eqref{eq:tp13} and \eqref{eq:H_2}, for
any $x\in \RR^d$, we have
\begin{equation}\label{eq2:mass_est}\begin{split}
H(1,x) & \le \int_{B(0,r)} \tp(1,x,y) \,dy + \frac{1}{r^{\delta}} \int_{B(0,r)^c} \tp(1,x,y) |y|^\delta \,dy \\
&\le \int_{B(0,r)} \tp(1,x,y) \, dy +M|x|^\delta\le \int_{B(0,r)} cH(1,3^{1/\alpha}x) \,dy  + M|x|^\delta,\\
&= \eta H(1,3^{1/\alpha}x) + M|x|^\delta,
\end{split}\end{equation}
where $M ={r^{-\delta}}$. Now, we can iterate the inequality
\eqref{eq2:mass_est} to obtain that for all $x\in \RR^d$,
\begin{align*}
H(1,x) &\le \eta H(1,3^{1/\alpha}x) + M|x|^\delta \\
&\le \eta \left[\eta H(1,3^{2/\alpha}x) + M|3^{1/\alpha}x|^\delta\right] + M|x|^\delta \\
&\le \eta^2 \left[ \eta H(1,3^{3/\alpha}x)+M |3^{2/\alpha}x|^\delta\right] + M(1+\eta 3^{\delta/\alpha})|x|^\delta\\
&\le \cdots\\
&\le  \eta^n H(1,3^{n/\alpha}x) + M[1+\eta 3^{\delta/\alpha} + \cdots + (\eta 3^{\delta/\alpha})^{n-1}]|x|^\delta.
\end{align*}
By \eqref{eq:H_2}, taking $n \to \infty$ in the inequality above, we
get that for any $x\in \RR^d$,
\begin{align*}
H(1,x) \le \frac{M}{1-\eta 3^{\delta/\alpha}} |x|^\delta,
\end{align*} yielding the desired assertion.
\end{proof}

Applying Proposition \ref{lem:mass_est} to \eqref{eq:tp13}, we immediately get

\begin{corollary}\label{eq:ub1}
There is a constant $C>0$ such that
$$
\tp(1,x,y) \le C (1 \land |x|^\delta) (1 \land |y|^\delta), \quad
x,y \in \RR^d.
$$
\end{corollary}

Next, we further refine upper bounds for $\tp(t,x,y)$.

\begin{lemma}\label{eq:1}
For any $t >0$ and  $x,y \in \RR^d$, we have
$$
\int_{B(y,|x-y|/2)} p(t,x,z)p(t,z,y) \,dz \le \frac{p(2t,x,y)}{2}.
$$
\end{lemma}
\begin{proof}
Fix $t>0$ and $x,y\in \RR^d$. By symmetry,
\begin{align*}
\int_{B(y,|x-y|/2)} p(t,x,z)p(t,z,y) \,dz = \int_{B(x,|x-y|/2)} p(t,x,z)p(t,z,y)\, dz.
\end{align*}
Hence, by \eqref{eq:CKp},
\begin{align*}
&2\int_{B(y,|x-y|/2)} p(t,x,z)p(t,z,y)\, dz\\ &= \int_{B(y,|x-y|/2)} p(t,x,z)p(t,z,y)\, dz + \int_{B(x,|x-y|/2)} p(t,x,z)p(t,z,y) \,dz\\
&\le \int_{\RR^d} p(t,x,z)p(t,z,y)\, dz = p(2t,x,y).
\end{align*} This completes the proof.
\end{proof}

\begin{lemma}\label{lem:intsplit}
There exists a constant $M>0$ such that for any  $t > 0$ and $x,y
\in \RR^d$, we have
$$
\tp(t,x,y) \le \int_{B(y,|x-y|/2)} \tp(t/2,x,z)\tp(t/2,z,y) \,dz + M h(t,x) p(t,x,y),
$$ where $h(t,x) = t^{-\delta/\alpha}|x|^\delta$.
\end{lemma}
\begin{proof}
By \eqref{eq:CKpt}, for any $t>0$
and $x,y\in \RR^d$, we have
\begin{align*}
\tp(t,x,y) &= \int_{B(y,|x-y|/2)} \tp(t/2,x,z)\tp(t/2,z,y) \,dz \\
&\quad+ \int_{B(y,|x-y|/2)^c} \tp(t/2,x,z)\tp(t/2,z,y) \,dz.
\end{align*}
Note that, according to \eqref{eq:H_1} and Proposition \ref{lem:mass_est}, \begin{align*}
\int_{\RR^d} \tp(t,x,y) \,dy =H(t,x)=H(1,t^{-1/\alpha}x)\le c h(t,x),\quad t>0,\;x\in \RR^d.
\end{align*}
For $t>0$ and $x,y,z\in \RR^d$ with $z \in B(y,|x-y|/2)^c$,  we have $$\tp(t/2,z,y) \le p(t/2,z,y) \le c_1 p(t,x,y).$$ Hence, for any $t>0$ and $x,y\in \RR^d$,
\begin{align*}
\int_{B(y,|x-y|/2)^c} \tp(t/2,x,z)\tp(t/2,z,y)\, dz &\le c_1  p(t,x,y)\int_{B(y,|x-y|/2)^c} \tp(t/2,x,z)\, dz \\
&\le c c_1 h(t/2,x) p(t,x,y)\\
& \le M h(t,x) p(t,x,y),
\end{align*}
thus we get the assertion of the lemma.
\end{proof}

\begin{theorem}\label{thm:UB}{\bf (Upper bounds)}
There is a constant $C>0$ such that for all $x, y \in \RR^d$,
\begin{align}\label{EQ:UB}
\tp(1,x,y) \le C (1\land |x|^\delta)(1\land |y|^\delta) p(1,x,y).
\end{align}
\end{theorem}

\begin{proof}  Let $\eta = 1/2$ and $\nu = 2^{(\delta-\alpha)/\alpha} <1$. As in Lemma \ref{lem:intsplit}, denote $h(t,x) = t^{-\delta/\alpha} |x|^\delta$. Note that
\begin{align}\label{eq0:UB}
\eta h(t/2,x) = \tfrac{1}{2} |(t/2)^{-1/\alpha}x|^\delta = 2^{(\delta-\alpha)/\alpha} |t^{-1/\alpha}x|^\delta = \nu h(t,x), \quad t>0,\, x\in\RR^d.
\end{align}
Let $M$ be the constant from Lemma \ref{lem:intsplit}. We will claim that for $n \ge 0$,
\begin{align}\label{e:proof}
\tp(t,x,y) \le [\eta^{n+1} + (1+\nu+\ldots+\nu^n)M h(t,x)] p(t,x,y),
\quad t>0, \; x,y\in \RR^d.
\end{align}
Indeed, for $t\in(0,1]$ and $x,y\in \RR^d$, by Lemmas \ref{eq:1} and
\ref{lem:intsplit},
\begin{align*}
\tp(t,x,y) \le [\eta + M h(t,x)]p(t,x,y),
\end{align*} where we used the fact $\tp(t,x,y)\le p(t,x,y)$ for any $t>0$ and $x,y\in \RR^d$.
Next, we use induction. Suppose that
\begin{align*}
\tp(t,x,y) \le [\eta^{n} + (1+\nu+\ldots+\nu^{n-1})M h(t,x)]
p(t,x,y),\quad t>0, \; x,y\in \RR^d.
\end{align*}
Then, for any $t>0$ and $x,y\in \RR^d$, by Lemma \ref{lem:intsplit},
Lemma \ref{eq:1} and \eqref{eq0:UB},
\begin{align*}
\tp(t,x,y) &\le \int_{B(y,|x-y|/2)} \tp(t/2,x,z)p(t/2,z,y)\, dz + M h(t,x)p(t,x,y) \\
&\le \int_{B(y,|x-y|/2)} \!\![\eta^{n} + (1+\nu+\ldots+\nu^{n-1})M h(t/2,x)] p(t/2,x,z)p(t/2,z,y) \,dz \\
&\quad + M h(t,x)p(t,x,y) \\
&\le  [\eta^{n} + (1+\nu+\ldots+\nu^{n-1})M h(t/2,x)] \eta p(t,x,y)  + M h(t,x)p(t,x,y) \\
&\le  [\eta^{n+1} + (\nu+\ldots+\nu^{n})M h(t,x)] p(t,x,y)  + M h(t,x)p(t,x,y) \\
&= [\eta^{n+1} + (1+\nu+\ldots+\nu^{n})M h(t,x)] p(t,x,y),
\end{align*}
and \eqref{e:proof} follows. Since $h(1,x) = |x|^\delta$, by letting $n$ to infinity in \eqref{e:proof}, we get
\begin{align}\label{e:proof1}
\tp(1,x,y) \le \frac{M}{1-\nu} |x|^\delta p(1,x,y), \quad x,y\in
\RR^d.
\end{align}

In the following, we pass to the proof of \eqref{EQ:UB}. By
symmetry, we may and do assume that $|x| \le |y|$. For $x,y\in
\RR^d$ with $|y|\ge |x| \ge 1$, \eqref{EQ:UB} follows by the
estimate $\tp(1,x,y) \le p(1,x,y)$. For $|x|\le |y|\le1$, we use
Corollary \ref{eq:ub1}  and the estimate that $p(1,x,y) \ge c$.
Finally, for $|x| <1 \le |y|$, \eqref{EQ:UB} follows by
\eqref{e:proof1}.
\end{proof}

\subsection{Lower bounds of $\tp(1,x,y)$}

We first begin with the following lemma, which is a consequence of Theorem \ref{thm:UB}.
\begin{lemma} There is a constant $C>0$ such that
\begin{align*}
H(1,x) \ge C(1 \land |x|^\delta), \quad x \in \RR^d.
\end{align*}
\end{lemma}
\begin{proof}
Let $R>0$ and $x \in B(0,R/2)$. Then
\begin{align*}
\int_{B(0,R)^c} |y|^\delta p(1,x,y) \,dy \le c_1\int_{B(0,R)^c} \frac{|y|^\delta}{|y|^{d+\alpha}}\, dy = c_2 R^{\delta-\alpha} \to 0 \quad \mbox{as } R \to \infty.
\end{align*}
Choose $R\ge 1$ be such that $c_2 C R^{\delta-\alpha} \le 1/2$, where $C$ is the constant given in \eqref{EQ:UB}.  Then, by \eqref{eq:0}, for $r \ge R\ge1$ and $x \in B(0,r/2)$, we have
\begin{equation}\label{e:lower} \begin{split}
\int_{\RR^d} \tp(1,x,y) \,dy &\ge r^{-\delta} \int_{B(0,r)} \tp(1,x,y) |y|^\delta \,dy \\
& = r^{-\delta} \left(|x|^\delta - \int_{B(0,r)^c} \tp(1,x,y) |y|^\delta \,dy \right) \\
& \ge r^{-\delta} \left(|x|^\delta -  C\int_{B(0,r)^c} |x|^\delta p(1,x,y) |y|^\delta \,dy \right)  \ge \frac{|x|^\delta}{2r^\delta}.
\end{split}\end{equation}
Hence, for $x \in B(0,R/2)$, by \eqref{e:lower},
\begin{align*}
\int_{\RR^d} \tp(1,x,y)\, dy \ge \frac{|x|^\delta}{2R^\delta} \ge \frac{|x|^\delta \land  1}{2R^\delta};
\end{align*}
for $x\in B(0,R/2)^c$, taking $r = 2|x|+1$ in \eqref{e:lower}, we can get that \begin{align*}
\int_{\RR^d} \tp(1,x,y) dy \ge \frac{|x|^\delta}{2(2|x|+1)^\delta}\ge  \frac{|x|^\delta}{2(4|x|)^\delta}\ge  \frac{1}{4^{\delta+1}}.
\end{align*} Combining both estimates above, we can prove the desired assertion.
\end{proof}

To obtain lower bounds of $\tp(t,x,y)$, we need to consider the difference between $p(t,x,y)$ and $\tp(t,x,y)$. Motivated by Duhamel's formula \eqref{e:for}, we define
\begin{align*}
p_1(t,x,y) = \int_0^t\int_{\RR^d}p(t-s,x,z) |z|^{-\alpha} p(s,z,y)
\,dz\,ds,\quad t>0,\; x,y\in \RR_0^d.
\end{align*} It is easy to see that $p_1(t,x,y)$ also enjoys the same scaling property as $p(t,x,y)$, i.e.,
\begin{align}\label{eq:scalp1}
    p_1(t,x,y)=t^{-d/\alpha} p_1(1,xt^{-1/\alpha},yt^{-1/\alpha}),\quad t>0,\; x,y\in \RR_0^d.
\end{align}

Let
\begin{align*}
G(t,x) = \int_0^t\int_{\RR^d}p(s,x,z) |z|^{-\alpha}\,dz\,ds, \quad
t>0,\, x \in \RR_0^d.
\end{align*}
By \cite[Lemma 2.3]{2017-KB-TG-TJ-DP},
\begin{align}\label{eq:Gest}
G(t,x)  \approx \log(1 + t|x|^{-\alpha}), \quad t>0, \, x \in
\RR_0^d.
\end{align}
\begin{lemma} \label{lem:p1est}
For all $x,y \in \RR_0^d$, we have
$$p_1(1,x,y) \approx[G(1,x) + G(1,y)]\, p(1,x,y).$$
\end{lemma}
\begin{proof}
By the 3P inequality (see \cite[(9)]{BJ} or \cite[(2.11)]{CKS2}),
for any $x,y\in \RR_0^d$,
\begin{align*}&\int_0^1\int_{\RR^d} p(1-s,x,z)|z|^{-\alpha}p(s,z,y)\,dz\,ds\\
&\le c_1p(1,x,y) \int_0^1\int_{\RR^d}(p(1-s,x,z)+p(s,z,y)) |z|^{-\alpha}\,dz\,ds \\
& = c_1 (G(1,x) + G(1,y) ) p(1,x,y)\,,
\end{align*}
thus we get the upper bound.

Now, we pass to the lower bound.
Since the function $|x|\mapsto\log(1 + |x|^{-\alpha})$ is decreasing in $|x|$, by \eqref{eq:Gest} and the symmetry of $p(t,x,y)$ and $p_1(t,x,y)$, it suffices to prove
\begin{align*}
p_1(1,x,y) \ge c_2 G(1,x)  p(1,x,y), \quad 0 < |x| \le |y|.
\end{align*}
First, let $|y|>4$  and $|x|<2$. Then, by \eqref{eq:pest},
\begin{align*}
\int_0^{1/2} \int_{B(0,3)} p(s,x,z) \frac{1}{|z|^\alpha}\, dz \, ds &\ge \int_0^{1/2} \int_{B(x,s^{1/\alpha})} p(s,x,z) \frac{1}{|z|^\alpha}\, dz \, ds \\
  &\ge c_0 \ge c_0 \int_0^{1/2} \int_{B(0,3)^c} p(s,x,z) \frac{1}{|z|^\alpha}\,  dz\, ds
\end{align*} and so
$$\int_0^{1/2} \int_{\RR^d} p(s,x,z) \frac{1}{|z|^\alpha}\, dz \, ds \approx \int_0^{1/2} \int_{B(0,3)} p(s,x,z) \frac{1}{|z|^\alpha}\, dz \, ds .$$
Hence,
\begin{equation}\label{eq1:p1est}\begin{split}
p_1(1,x,y) &\ge \int_0^{1/2} \int_{B(0,3)} p(s,x,z) \frac{1}{|z|^\alpha} p(1-s,z,y)\, dz\, ds\\
&\approx \int_0^{1/2} \int_{B(0,3)} p(s,x,z) \frac{1}{|z|^\alpha} p(1,x,y) \,dz \,ds\\
&\approx \int_0^{1/2} \int_{\RR^d} p(s,x,z) \frac{1}{|z|^\alpha} p(1,x,y)\, dz \,ds\\
& = G(1/2,x) p(1,x,y)\approx G(1,x) p(1,x,y),
\end{split}\end{equation} where in the second step we used the fact that
$$p(1-s,z,y)\approx p(1,x,y),\quad 0<s\le 1/2, |x|<2, |z|\le 3, |y|>4.$$
Next, suppose that $|x| \le 1$ and $|x| \le |y| \le 4$. Then, $p(1,x,y) \approx c$.
Note that
\begin{align*}
\int_{\RR^d} p_1(t,x,z) p(r,z,y)\,dz &= \int_{\RR^d} \int_0^t\int_{\RR^d}p(s,x,w) |w|^{-\alpha} p(t-s,w,z) p(r,z,y)  \,dw\,ds\, dz \\
&= \int_0^t\int_{\RR^d}p(s,x,w) |w|^{-\alpha} p(t+r-s,w,y)  \,dw\,ds \\
& \le p_1(t+r,x,y)\,.
\end{align*}
Hence, by the scaling property of $p_1(t,x,y)$ and \eqref{eq1:p1est},
\begin{align*}
p_1(1,x,y) &\ge \int_{B(0,8)^c} p_1(1/2^\alpha,x,z) p(1-1/2^\alpha,z,y)\,dz\\
 &= \int_{B(0,8)^c} 2^d p_1(1,2x,2z) p(1-1/2^\alpha,z,y)\,dz\\
& \ge c_3 G(1,2x) \int_{B(0,8)^c} p(1,2x,2z) p(1-1/2^\alpha,z,y)\,dz\\
&\ge c_4 G(1,x) \ge c_5 G(1,x) p(1,x,y),
\end{align*} where in the third inequality we used the fact that
\begin{align*}\int_{B(0,8)^c} p(1,2x,2z) p(1-1/2^\alpha,z,y)\,dz\ge &c_6\int_{B(0,8)^c} \frac{1}{|z-x|^{d+\alpha} |z-y|^{d+\alpha}}\,dz\\
\ge &c_7\int_{B(0,8)^c}\frac{1}{|z|^{2d+2\alpha}}\,dz \ge c_8.\end{align*}
At last, suppose that $1 \le |x| \le |y|$. Then $G(1,x) \approx |x|^{-\alpha}$. Hence,
\begin{align*}
p_1(1,x,y) &\ge \int_0^{1/2} \int_{B(|x|,1/2)} p(s,x,z) \frac{1}{|z|^\alpha}p(1-s,z,y)\,dz\,ds\\
&\ge c_9\int_0^{1/2} \int_{B(|x|,1/2)} p(s,x,z) \frac{1}{|x|^\alpha}p(1-s,z,y)\,dz \,ds\\
&\ge c_{10} G(1,x) p(1,x,y),
\end{align*} where the last inequality follows from the facts that
$$p(1-s,z,y)\approx p(1,x,y),\quad 0<s<1/2, z\in B(x,1/2), 1 \le |x| \le |y|$$ and
$$\int_0^{1/2} \int_{B(|x|,1/2)} p(s,x,z) \,dz\,ds \ge c_{11}.$$ The proof is complete.
\end{proof}
The following estimate is generally well known (see e.g. \cite[Section 6]{2008-KB-WH-TJ-sm} for further background).
\begin{lemma}\label{lem:tpexp}
For all $t>0$ and $x,y \in \RR^d_0$, we have
\begin{align*}
\tp(t,x,y) \ge p(t,x,y) \exp\left[\kappa\, \frac{p_1(t,x,y)}{p(t,x,y)}\right].
\end{align*}
\end{lemma}
\begin{proof} Since the proof is a little long, we will postpone it to the appendix.
\end{proof}

We note that the estimate in Lemma \ref{lem:tpexp} is not sharp.
More precisely, one may show that $\lim\limits_{s \to 0^+}
\frac{p_1(1,sx,y)}{-\ln (s|x|) p(1,sx,y)} =
\frac{\Gamma(\frac{d-\alpha}{2})}{2^{\alpha-1}\Gamma(\frac{\alpha}{2})\Gamma(\frac{d}{2})}$.
Hence, by \eqref{EQ:main}, for fixed  $y \not=0$,  $\tp(1,x,y)
e^{-\kappa\, \frac{p_1(1,x,y)}{p(1,x,y)}} \to \infty$ as $x \to 0$.
However, we still can get the following useful estimate.

\begin{corollary}\label{cor:lowerb}
There are constants $c,\gamma>0$  such that for all $t>0$ and $x,y \in \RR^d_0$, we have
\begin{align}\label{eq:lowerb}
\tp(t,x,y) \ge  c \left[1 \land (t^{-1/\alpha}|x|)^\gamma\right]\left[1 \land (t^{-1/\alpha}|y|)^\gamma\right] p(t,x,y)
\end{align}
\end{corollary}

\begin{proof}
Lemmas \ref{lem:tpexp} and  \ref{lem:p1est} along with \eqref{eq:Gest} yield that for any $x,y\in \RR_0^d$,
\begin{align*}
\tp(t,x,y) &\ge p(t,x,y) \exp\left[\kappa\, \frac{p_1(t,x,y)}{p(t,x,y)}\right]\\
&\ge p(t,x,y) \exp\left[-c (G(t,x) + G(t,y))\right]\\
& \ge p(t,x,y) \exp\left[- C (\log(1 + t|x|^{-\alpha}) + \log(1 + t|y|^{-\alpha})) \right]\\
& = p(t,x,y) \left[(1 + t|x|^{-\alpha})^{-C}\right] \left[(1 + t|y|^{-\alpha})^{-C} \right]\\
& \ge p(t,x,y) \left[2^{-C}(1 \vee t|x|^{-\alpha})^{-C}\right] \left[2^{-C}(1 \vee t|y|^{-\alpha})^{-C} \right],
\end{align*}
thus we get \eqref{eq:lowerb} with $\gamma = \alpha C$ and $c = 4^{-C}$.
\end{proof}

\begin{lemma}\label{lem:lowerCase1}
For any $r>0$, there is a constant $C_r>0$ such that for all $x,y\in \RR^d$ with $|x|\land |y|\ge r$,
\begin{align*}
\tp(1,x,y) \ge C_r p(1,x,y).
\end{align*}
\end{lemma}

\begin{proof}
For $r>0$ and $x,y\in \RR^d$ with $|x| \land |y| \ge r$, by \eqref{eq:lowerb}, we get
\begin{align*}
\tp(1,x,y) \ge c (1 \land r)^{2\gamma} p(1,x,y),
\end{align*}
where $c$ and $\gamma$ are the constants from Corollary \ref{cor:lowerb}.
\end{proof}

\begin{lemma}\label{lem:lowerCase2}
For any $R>0$, there is a constant $C_R>0$ such that for any $x,y\in
\RR^d$ with $|x| \vee |y| \le R$,
$$
    \tp(1,x,y)\ge C_R |x|^\delta |y|^\delta.
$$
\end{lemma}
\begin{proof}
By \eqref{e:lower}, there exists a constant $R_0\ge 2 \cdot 3^{1/\alpha}$ large enough such that for all $x\in B(0,R_0/2)$,
$$
\int_{B(0,R_0)} \tp(1,x,y)\,dy\ge R_0^{-\delta}\int_{B(0,R_0)}\tp(1,x,y) |y|^\delta\,dy \ge \frac{|x|^\delta}{2 R_0^\delta}.
$$
On the other hand, by \eqref{EQ:UB}, for all $r_0>0$ and $x\in \RR^d$,
$$
\int_{B(0,r_0)}\tp (1,x,y)\,dy\le C_1 |x|^\delta \int_{B(0,r_0)}p(1,x,y)|y|^\delta\,dy\le C_1|x|^\delta r_0^\delta.
$$
We take $r_0=\frac{1}{(4C_1)^{1/\delta} R_0}$. For $0<a<b<\infty$, let $D(a,b) = B(0,b) \setminus B(0,a)$. Then, for $|x| < R_0/2$,
\begin{align}\label{eq:lowerCase2}
\int_{D(r_0,R_0)} \tp(1,x,z) dz \ge |x|^\delta\left(\frac{1}{2R_0^\delta} - C_1 r_0^\delta\right) = \frac{|x|^{\delta}}{4R_0^\delta}.
\end{align}
Therefore, by \eqref{eq:CKpt} and \eqref{eq:lowerCase2}, for all $x,y\in \RR^d$ with
$|x|\vee |y|\le R_0/2$,
\begin{align*}\tp(3,x,y)&\ge \int_{D(r_0,R_0)}\int_{D(r_0,R_0)}\tp (1, x,z) \tp(1,z,w)\tp(1,w,y)\,dz\,dw\\
&\ge  \frac{|x|^\delta |y|^\delta}{16 R_0^{2\delta}}  \inf_{z,w\in D(r_0, R_0)}\tp(1,z,w).
\end{align*}
Next, by Lemma \ref{lem:lowerCase1},
\begin{equation}\label{e:rem}
\inf_{z,w\in D(r_0,R_0)}\tp(1,z,w) \ge C_{r_0}\inf_{z,w\in D(r_0,R_0)} p(1,z,w) \ge  \frac{c C_{r_0}}{(2R_0)^{d+\alpha}} >0,
\end{equation} where $C_{r_0}>0$ is a constant given in Lemma \ref{lem:lowerCase1}.
Hence, $$\tp(3,x,y) \ge c_0 |x|^\delta |y|^\delta,\quad |x| \vee |y| < R_0/2.$$
Now, by the scaling property of $\tp$, we obtain
\begin{align*}
\tp(1,x,y) = 3^{-d/\alpha} \tp(3,3^{-1/\alpha}x, 3^{-1/\alpha}y) \ge c_0 3^{-(d+2\delta)/\alpha} |x|^\delta |y|^\delta, \quad |x| \vee |y| \le \frac{R_0}{2 \cdot 3^{1/\alpha}}.
\end{align*} This completes the proof.
\end{proof}
\begin{remark}Instead of applying Lemma \ref{lem:lowerCase1}, we can make use of the Feynman-Kac formula \eqref{e:hk2} for the semigroup $(\tilde P_t)_{t\ge0}$ and Dirichlet heat kernel estimates for fractional Laplacian obtained in \cite{CKS} to achieve \eqref{e:rem}. \end{remark}

\begin{theorem}\label{Thm:LB}{\bf (Lower bounds)}\,\, There is a constant $C>0$ such that for all $x,y \in \RR^d$,
\begin{align*}
\tp(1,x,y) \ge C(1\land |x|^\delta)(1\land |y|^\delta) p(1,x,y).
\end{align*}
\end{theorem}

\begin{proof}
By symmetry, we will consider only $|x| \le |y|$. For $|w_1| \le 1/4$, $|w_2| > 1$ and  $1/4 \le |z| \le 1/2$, by Lemmas \ref{lem:lowerCase2} and \ref{lem:lowerCase1}, we have
$$
\tp(1,w_1,z) \ge c_1 |w_1|^\delta $$ and
$$\tp(1,w_2,z) \ge c_2 p(1,w_2,z)\approx p(2,w_2,z)\approx p(2,w_2,w_1).$$
Hence, for any $|w_1|\le 1/4$ and $|w_2|>1$,
\begin{equation}\label{EQ:lower1}\begin{split}
\tp(2,w_1,w_2) &\ge \int_{B(0,1/2) \setminus B(0,1/4)} \tp(1,w_1,z) \tp(1,z,w_2) \,dz \ge c |w_1|^\delta p(2,w_1,w_2).
\end{split}\end{equation}
Therefore, for $|x| \le 2^{-1/\alpha}/4$ and $|y| > 2^{-1/\alpha}$, by \eqref{EQ:lower1},
\begin{align*}
\tp(1,x,y) = &2^{d/\alpha} \tp(2,2^{1/\alpha}x, 2^{1/\alpha}y) \ge c
2^{d/\alpha} |2^{1/\alpha}x|^\delta
p(2,2^{1/\alpha}x,2^{1/\alpha}y)\\ = & c 2^{\delta/\alpha}
|x|^\delta p(1,x,y).
\end{align*}
Next, for $|x| \land |y| \ge 2^{-1/\alpha}/4$, we use Lemma \ref{lem:lowerCase1}. Finally, for $|x| \vee |y| \le 2^{-1/\alpha}$, we apply Lemma \ref{lem:lowerCase2}.
\end{proof}

Two-sided estimates for $\tp(t,x,y)$ stated in Theorem \ref{th:main} is a direct consequence of the scaling property of $\tp(t,x,y)$ and Theorems \ref{thm:UB} and \ref{Thm:LB}.

\subsection{Joint continuity of $\tp(t,x,y)$} To prove the joint continuity of $\tp(t,x,y)$, we just follow the same argument of \cite[Subsection 4.3]{2016-KB-BD-PK-pa}. For the sake of completeness, we present the proof here.
\begin{lemma}\label{L:con1}For any fixed $x\in \RR_0^d$, the function $\RR^d_0\ni y \mapsto \tp(t,x,y)$ is continuous. \end{lemma}
\begin{proof} Fix $x,y,z\in \RR_0^d$ with $z\to y$. Then, by \eqref{e:for},
\begin{align*}\tp (1,x,y)-\tp (1,x,z)=&p(1,x,y)-p(1,x,z)\\
&+\int_0^1\int_{\RR^d} \tp(1-s,x,w)q(w)(p(s,w,y)-p(s,w,z))\,dw\,ds.
\end{align*}

For any $\varepsilon>0$ small enough, by \eqref{eq:Gest},
\begin{equation}\label{e:con1}\begin{split}&-\int_0^\varepsilon \int_{\RR^d} \tp(1-s,x,w)q(w)p(s,w,y)\,dw\,ds\\
&\le -\int_0^\varepsilon \int_{\RR^d} p(1-s,x,w)q(w)p(s,w,y)\,dw\,ds\\
&\le-c_1\int_0^\varepsilon \int_{\RR^d} p(s,w,y)q(w)\,dw\,ds\\
&= -\kappa c_1 G(\varepsilon,y) \le c_2 \varepsilon |y|^{-\alpha}.
\end{split}\end{equation} Similarly, we have
\begin{align*}-\int_0^\varepsilon \int_{\RR^d} \tp(1-s,x,w)q(w)p(s,w,z)\,dw\,ds\le c_2\varepsilon |z|^{-\alpha}.\end{align*}

For any $\varepsilon\le s \le 1$ and $w,y,z\in \RR^d$ with $z\to y$, we have $p(s,w,y)\asymp p(s,w,z)$. By the dominated convergence theorem, it holds that
$$\int_\varepsilon^1\int_{\RR^d} \tp(1-s,x,w)q(w)(p(s,w,y)-p(s,w,z))\,dw\,ds \to0,\quad z\to y.$$

Combining with all the estimates above, we prove the desired assertion. \end{proof}

\begin{prop}\label{pro-con} The function $\tp(t,x,y)$ is jointly continuous with respect to $t>0$ and $x,y \in \RR^d_0.$\end{prop}
\begin{proof} By the scaling property of $\tp(t,x,y)$, it suffices to show the continuity of $\tp(1,x,y)$ with respect to $x,y\in \RR^d_0$. As indicated in the proof of Lemma \ref{L:con1}, we only need to verify that
$$-\int_0^1\int_{\RR^d} |\tp(1-s,\tilde x,w)p(s,w,\tilde y)-\tp(1-s,x,w)p(s,w,y)|q(w)\,dw\,ds\to 0$$ for any $x,y,\tilde x, \tilde y\in \RR_0^d$ with $\tilde x\to x$ and $\tilde y\to y$.

In addition to \eqref{e:con1}, we have
\begin{align*}&-\int_{1-\varepsilon}^1\int_{\RR^d} \tp(1-s,x,w)q(w)p(s,w,y)\,dw\,ds\\
&=-\int_0^\varepsilon \int_{\RR^d} \tp(s,x,w)q(w)p(1-s,w,y)\,dw\,ds\\
&\le - \int_0^\varepsilon \int_{\RR^d} p(s,x,w)q(w)p(1-s,w,y)\,dw\,ds \\
&\le c_1\varepsilon |x|^{-\alpha}. \end{align*}

For any $\varepsilon<s<1-\varepsilon$ and $x,y,z,\tilde x,\tilde y\in \RR_0^d$ with $x\to \tilde x$ and $y\to \tilde y$, $p(s,z,\tilde y)\approx p(s,z,y),$ and
$ \tilde p(1-s,\tilde x,z)\approx p(1-s,x,z),$ thanks to  Lemma \ref{L:con1}. Then, by the dominated convergence theorem, it holds that
$$-\int_\varepsilon^{1-\varepsilon}\int_{\RR^d} |\tp(1-s,\tilde x,w)p(s,w,\tilde y)-\tp(1-s,x,w)p(s,w,y)|q(w)\,dw\,ds\to 0$$ for any $x,y,\tilde x, \tilde y\in \RR_0^d$ with $\tilde x\to x$ and $\tilde y\to y$.

Hence, according to all the estimates above, we prove the desired assertion.
\end{proof}

\begin{theorem}\label{thm:cont} {\bf (Joint continuity)}\,\, The function $\tp(t,x,y)$ is jointly continuous with respect to $t>0$ and $x,y \in \RR^d.$\end{theorem}
\begin{proof}According to Proposition \ref{pro-con} and the scaling property of $\tp(t,x,y)$, we only need to verify that $p(1,x,y)$ is jointly continuous with respect to $x,y\in \RR^d$ when $x=0$ or $y=0$. Since $\tp(1,x,y)=0$ when $x=0$ or $y=0$, the desired assertion for the joint continuity is a direct consequence of the fact that $\tp(1,x,y)\ge0$ and two-sided estimates for $\tp(1,x,y)$ on $\RR_0^d\times \RR_0^d$.    \end{proof}

\subsection{Dirichlet forms}
Finally, we discuss the Dirichlet form associated with the
Schr\"{o}dinger operator $\mathcal{L}$ given by \eqref{e:op1}; see
\cite{FOT} for the theory of Dirichlet forms. According to
\cite[Theorem 2.5]{DC}, the Feynman-Kac semigroup $(\tilde
P_t)_{t\ge0}$ in $L^2(\RR^d_0;dx)$ coincides with the semigroup
corresponding to $\tilde\E$ with the domain
$$
\mathscr{D}(\tilde \E)=\Big\{f\in L^2(\RR_0^d;dx):
\iint_{\RR^d_0\times \RR_0^d}
\frac{(f(x)-f(y))^2}{|x-y|^{d+\alpha}}\,dx\,dy+\int_{\RR^d_0}
f^2(x)|q(x)|\,dx<\infty\Big\}
$$
and defined by
\begin{align*}
\tilde \E (f,g)=&\frac{1}{2}\iint_{\RR^d_0\times \RR_0^d}
{(f(x)-f(y))(g(x)-g(y))}\nu(x-y)\,dx\,dy\\
&+\int_{\RR^d_0} f(x)g(x)|q(x)|\,dx
\end{align*}
for any $f,g\in \mathscr{D}(\tilde \E),$ where $\nu$ is defined by \eqref{e:levymea}. Clearly, the quadratic form $(\tilde \E,
\mathscr{D}(\tilde \E) )$ is equivalently given by
\begin{align*}
\tilde \E (f,g)=&\frac{1}{2}\iint_{\RR^d\times \RR^d}
{(f(x)-f(y))(g(x)-g(y))}\nu(x-y)\,dx\,dy\\
&+\int_{\RR^d} f(x)g(x)|q(x)|\,dx,\\
\mathscr{D}(\tilde \E)=&\Big\{f\in L^2(\RR^d;dx): \tilde
\E(f,f)<\infty\Big\},
\end{align*}which are extended to be defined on $L^2(\RR^d;dx).$

\begin{prop} $(\tilde \E, \mathscr{D}(\tilde \E) )$ is a symmetric regular Dirichlet form on $L^2(\RR^d;dx)$ with core $C_c^\infty(\RR^d).$ \end{prop}
\begin{proof} Define
$$\E(f,g)=\frac{1}{2}\iint_{\RR^d\times \RR^d} (f(x)-f(y))(g(x)-g(y))\nu(x-y)\,dx\,dy,\quad f,g \in \mathscr{D}(\E)$$ and
$$\mathscr{D}(\E)=\Big\{f\in L^2(\RR^d;dx):\E(f,f)<\infty\Big\}.$$ Then, $(\E,\mathscr{D}(\E))$
is a symmetric Dirichlet form on $L^2(\RR^d;dx)$ associated with fractional Laplacian; moreover, $C_c^\infty(\RR^d)\subset \mathscr{D}(\Delta^{\alpha/2})$ (here $\mathscr{D}(\Delta^{\alpha/2})$ denotes the $L^2$-domain of $\Delta^{\alpha/2}$ on $L^2(\RR^d;dx)$) and $(\E,\mathscr{D}(\E))$ is regular with
core $C_c^\infty(\RR^d)$; see \cite[Section 2.2.2]{CF} for more details. On the other hand, due to $\alpha<d$, we can verify that
$$\int_{\RR^d} f(x)^2|q(x)|\,dx<\infty,\quad f\in C_c^\infty(\RR^d).$$ In particular,
$C_c^\infty(\RR^d)\subset \mathscr{D}(\Delta^{\alpha/2})\cap
L^2(\RR^d;|q(x)|\,dx)\subset \mathscr{D}(\tilde \E)$.

It is easy to prove that $(\tilde \E, \mathscr{D}(\tilde \E) )$ is a
symmetric  Dirichlet form on $L^2(\RR^d;dx)$. Next, we claim that
$C_c^\infty(\RR^d)$ is dense in $\mathscr{D}(\tilde\E)$ with the
norm $\sqrt{\tilde\E} +\|\cdot\|_{L^2(\RR^d;dx)}.$ According to the
Hardy inequality for fractional Laplacian (see \cite[Proposition
5]{2016-KB-BD-PK-pa}), there is a constant $C_0>0$ such that for all
$f\in L^2(\RR^d;dx)$,
$$\int_{\RR^d} f^2(x)|q(x)|\,dx\le C_0 \E(f,f).$$ Thus, the norms $\sqrt{\tilde\E} +\|\cdot\|_{L^2(\RR^d;dx)}$ and $\sqrt{\E} +\|\cdot\|_{L^2(\RR^d;dx)}$ are equivalent.
Therefore, the desired assertion above immediately follows from the
fact that $(\E,\mathscr{D}(\E))$ is regular with core
$C_c^\infty(\RR^d)$.
 \end{proof}

Let $h(x) = |x|^\delta$, and define
\begin{align*}
\bar{\cE}(f,f) = \frac{1}{2}\iint_{\Rd\times\Rd}
\left(\frac{f(x)}{h(x)} - \frac{f(y)}{h(y)}\right)^2 h(x)h(y)
\nu(x-y)\, dx\, dy, \quad f\in\mathscr{D}(\bar\cE),
\end{align*}
where $\mathscr{D}(\bar\cE) = \{f \in L^2(\RR^d; dx) \colon
\bar{\cE}(f,f)< \infty\}$.

\begin{prop}
We have $\mathscr{D}(\tilde\cE) = \mathscr{D}(\bar\cE)$ and$$
\tilde\cE(f,f) = \bar{\cE}(f,f), \quad   f \in
\mathscr{D}(\tilde\cE).
$$
\end{prop}
\begin{proof} Denote by $\langle\cdot,\cdot\rangle$ the inner
product of $L^2(\Rd;dx)$. Recall that $(\tilde P_t)_{t\ge0}$ is well
defined on $L^2(\Rd;dx)$ by setting $\tilde P_tf(0)=0$ for any $f\in
L^2(\Rd;dx)$. According to \eqref{eq:00},
\begin{align*}
\langle f-\tilde{P}_t f,f \rangle &=  \int_{\Rd} \left(f(x) - \int_{\Rd} \tp(t,x,y) f(y)\, dy\right)  f(x)\, dx \\
&=  \int_{\Rd} \left( \frac{f(x)}{h(x)}\int_\Rd \tp(t,x,y) h(y) \,dy - \int_{\Rd} \tp(t,x,y)h(y) \frac{f(y)}{h(y)}\, dy\right)  f(x)\, dx \\
&=  \iint_{\Rd\times\Rd} \tp(t,x,y)\left( \frac{f(x)}{h(x)} -
\frac{f(y)}{h(y)} \right) \frac{f(x)}{h(x)}  h(x) h(y) \,dx \,dy.
\end{align*}
Hence, by the symmetry,
\begin{align*}
\langle f-\tilde{P}_t f,f \rangle &= \frac{1}{2}
\iint_{\Rd\times\Rd} \tp(t,x,y)\left( \frac{f(x)}{h(x)} -
\frac{f(y)}{h(y)} \right)^2 h(x) h(y)\,dx \,dy
\end{align*}
Note that
\begin{align}\label{eq1:DF}
\lim_{t \to 0^+} \frac{\tp(t,x,y)}{t} = \nu(x-y),\quad x,y\in
\RR_0^d.
\end{align}
Indeed, by the Duhamel formula \eqref{e:for}, for any $x,y\in
\RR_0^d$,
\begin{align*}
\lim_{t \to 0^+} \frac{\tp(t,x,y)}{t} &= \lim_{t \to 0^+} \frac{p(t,x,y)}{t} + \lim_{t \to 0^+} \frac{1}{t} \int_0^t \int_\Rd \tp(t-s,x,z) q(z) p(s,z,y)\, dz\,  ds\\
&= \nu(x-y) + \lim_{t \to 0^+} \frac{1}{t} \int_0^t \int_\Rd
\tp(t-s,x,z) q(z) p(s,z,y)\, dz \, ds.
\end{align*}
Next, by the fact that $\tp(t,x,y)\le p(t,x,y)$ for all $t>0$ and
$x,y\in \RR^d$, Lemma \ref{lem:p1est}, \eqref{eq:scalp1} and
\eqref{eq:Gest}, for any $x,y\in \RR_0^d$,
\begin{align*}
\frac{1}{t}\int_0^t \int_\Rd \tp(t-s,x,z) |z|^{-\alpha} p(s,z,y)
\,dz \,ds \le c_1t^{-1}{p(t,x,y)}[G(t,x)+G(t,y)]
\stackrel{t\to0}{\longrightarrow} 0,
\end{align*}
and so we  get \eqref{eq1:DF}.

Now, for $f \in \mathscr{D}(\tilde\cE)$, by Fatou's Lemma and
\eqref{eq1:DF}, we have
\begin{align*}
\tilde\cE(f,f) &= \lim_{t \to 0} \frac{1}{t}\langle f-\tilde{P}_t f,f \rangle \\
&= \lim_{t \to 0}\frac{1}{2} \iint_{\Rd\times\Rd} \frac{\tp(t,x,y)}{t}\left( \frac{f(x)}{h(x)} - \frac{f(y)}{h(y)} \right)^2  h(x) h(y)\,dy\,   dx\\
&\ge \frac{1}{2} \iint_{\RR^d_0\times\RR^d_0} \liminf_{t \to 0}\frac{\tp(t,x,y)}{t}\left( \frac{f(x)}{h(x)} - \frac{f(y)}{h(y)} \right)^2  h(x) h(y)\,dy \,  dx\\
&= \frac{1}{2} \iint_{\RR^d_0\times \RR^d_0} \left(
\frac{f(x)}{h(x)} - \frac{f(y)}{h(y)} \right)^2  h(x) h(y) \nu(x-y)
\, dy\, dx = \bar\cE(f,f).
\end{align*}
Hence, $\mathscr{D}(\tilde\cE) \subset \mathscr{D}(\bar\cE)$. On the
other hand, we take $f \in \mathscr{D}(\bar\cE)$. Since $\tp(t,x,y)
\le p(t,x,y) \le c_2t\nu(x-y)$ for all $t>0$ and $x,y\in \RR^d$,
according to the dominated convergence theorem and \eqref{eq1:DF}
again, we have
\begin{align*}
\tilde\cE(f,f) &= \lim_{t \to 0} \frac{1}{t}\langle f-\tilde{P}_t f,f \rangle \\
&= \lim_{t \to 0}\frac{1}{2} \iint_{\Rd\times\Rd} \frac{\tp(t,x,y)}{t}\left( \frac{f(x)}{h(x)} - \frac{f(y)}{h(y)} \right)^2  h(x) h(y)\,dy\,   dx\\
&= \frac{1}{2} \iint_{\Rd\times\Rd} \left( \frac{f(x)}{h(x)} -
\frac{f(y)}{h(y)} \right)^2  h(x) h(y) \nu(x-y) \, dy\,
dx=\bar\cE(f,f).
\end{align*}
Combining with both inequalities, we prove the desired assertion.
\end{proof}

\begin{remark}The construction of $(\bar \E, \mathscr{D}(\bar \E))$ can be deduced from Doob's theory of $h$-transformations; see \cite[Chapter 11]{CW} for more details. Indeed, as shown by \eqref{e:eee}, the function $h=|x|^\delta$ is harmonic with respect to the operator $\mathcal{L}$ given by \eqref{e:op1}. Define $\mathcal{L}_hf(x):=h(x)^{-1} \mathcal{L}(fh)(x)$ for all $f\in L^2(\RR^d;h(x)^2\,dx)$. It is easy to see that the operator $\mathcal{L}_h$ is symmetric on $L^2(\RR^d;h(x)^2\,dx)$, and the associated symmetric regular Dirichlet form $(\E_h, \mathscr{D}(\E_h))$ on $L^2(\RR^d;h(x)^2\,dx)$ is given by
\begin{align*}\E_h(f,f)=&-\langle \mathcal{L}_h f,f\rangle_{L^2(\RR^d;h(x)^2\,dx)}\\
=&-\iint_{\RR^d\times \RR^d} \left(f(y)h(y)-f(x)h(x)\right)f(x)h(x)\nu(x-y)\,dy\,dx\\
&-\kappa\int_{\RR^d} |x|^{-\alpha} f(x)^2h^2(x)\,dx \\
=&-\iint_{\RR^d\times \RR^d} \left(f(y)-f(x)\right)f(x)h(y)h(x)\nu(x-y)\,dy\,dx\\
&-\iint_{\RR^d\times \RR^d}\left(h(y)-h(x)\right) f(x)^2h(x)\nu(x-y)\,dy\,dx\\
&-\kappa\int_{\RR^d} |x|^{-\alpha} f(x)^2h^2(x)\,dx\\
=&-\iint_{\RR^d\times \RR^d} \left(f(y)-f(x)\right)f(x)h(y)h(x)\nu(x-y)\,dy\,dx\\
=&\frac{1}{2}\iint_{\RR^d\times \RR^d} \left(f(y)-f(x)\right)^2h(y)h(x)\nu(x-y)\,dy\,dx\end{align*} for all $f\in \mathscr{D}(\E_h)$, where in the fourth equality we used \eqref{e:eee} and the last equality follows form the property that $\nu(x-y)=\nu(y-x)$.
Note that
$$\E_h(f,f)=-\langle \mathcal{L}_h f,f\rangle_{L^2(\RR^d;h(x)^2\,dx)}=- \langle \mathcal{L} (h f),hf\rangle_{L^2(\RR^d;dx)}=\tilde \E(hf,hf).$$
Combining both equalities above together, we arrive at
$$\tilde \E(f,f)=\E_h(fh^{-1},fh^{-1})=\frac{1}{2}\iint_{\RR^d\times \RR^d} \left(\frac{f(y)}{h(y)}-\frac{f(x)}{h(x)}\right)^2h(y)h(x)\nu(x-y)\,dy\,dx.$$ The  right side of the equality above coincides with the expression of $(\bar\cE, \mathscr{D}(\bar\cE)).$ \end{remark}

\section{Appendix: Proof of Lemma \ref{lem:tpexp}}
Let $q_0(x)=-|x|^{-\alpha}$. For any $\lm\ge 0$, denote by $p^{\lm}(t,x,y)$ the heat kernel associated with the generator $\Delta^{\alpha/2} + \lambda q_0(x)$. Hence, by Duhamel's formula (see \cite[Propositions 5.2 and 5.3]{DC} and their proofs), we have
$$
p^{\lambda}(t,x,y) = p^\nu(t,x,y) + (\lm-\nu) \int_0^t \int_{\RR_0^d} p^\nu(t-s,x,z) q_0(z) p^{\lm}(s,z,y)\, dz\, ds
$$ for any $t>0$ and $x,y\in \RR_0^d$.
Noting that $q_0(x)<0$ for all $x\in \RR_0^d$, we can rewrite the equality above as
$$
p^\nu(t,x,y) =  p^{\lambda}(t,x,y) + \int_0^t \int_{\RR_0^d} p^\nu(t-s,x,z) (\lm-\nu) |q_0(z)| p^{\lm}(s,z,y)\, dz\, ds.$$
For any $\lm \ge0$, $t>0$ and $x,y\in \RR_0^d$, we set
\begin{align*}p_0^{\lm}(t,x,y) &= p^{\lm}(t,x,y),\\
p_n^{\lm}(t,x,y) &= \int_0^t \int_{\RR_0^d} p_{n-1}^{\lm}(t-s,x,z)
|q_0(z)| p^{\lm}(s,z,y)\, dz\,ds\,, \quad n \ge  1.\end{align*}
Then, by \cite[Lemma 1 and the proof of Theorem
2]{2008-KB-WH-TJ-sm},
$$
p_{n+m+1}^{\lm}(t,x,y) = \int_0^t \int_{\RR_0^d} p_{n}^{\lm}(t-s,x,z)  |q_0(z)| p_m^{\lm}(s,z,y)\, dz\,ds\,, \quad m,n \ge  0
$$ and
\begin{align*}
p^{\eta-\lm}(t,x,y) = \sum_{n=0}^\infty \lm^n p_n^\eta(t,x,y),\quad \eta > \lm >0.
\end{align*}
Furthermore, we have

\begin{lemma}
Let $0 < \lm < \eta< \infty$. For all $x,y\in \RR_0^d$ and $t>0$,
\begin{equation}\label{eq:pkv_rel}
\sum_{n=k}^\infty \binom{n}{k}\lm^{n-k} p_n^{\eta}(t,x,y) = p_k^{\eta - \lm}(t,x,y)\,, \quad k =0,1,2,\cdots.
\end{equation}
\end{lemma}
\begin{proof}
We use induction. When $k=0$, \eqref{eq:pkv_rel} holds trivially. For $k=1$,
\begin{align*}
p_1^{\eta - \lm}(t,x,y) &=   \int_{0}^t \int_{\RR_0^d} p^{\eta-\lm}(t-s,x,z) |q_0(z)| p^{\eta-\lm}(s,z,y)\, dz\,ds \\
& =  \int_{0}^t \int_{\RR_0^d}  \sum_{i=0}^\infty \lm^i p_i^{\eta}(t-s,x,z) |q_0(z)| \sum_{j=0}^\infty \lm^j p_j^{\eta}(s,z,y)\, dz\, ds\\
& = \sum_{i=0}^\infty \sum_{j=0}^\infty \lm^{i+j}  p_{i+j+1}^{\eta}(t,x,y) = \sum_{n=1}^\infty n \lm^{n-1}  p_n^{\eta}(t,x,y)\,.
\end{align*}
Next, we assume that (\ref{eq:pkv_rel}) holds for some $k \in
\NN_+$. We get
\begin{align*}
p_{k+1}^{\eta-\lm}(t,x,y) &=  \int_{0}^t \int_{\RR_0^d} p_k ^{\eta-\lm}(t-s,x,z) |q_0(z)| p^{\eta-\lm}(s,z,y)\, dz\,ds \\
& =   \int_{0}^t \int_{\RR_0^d}  \sum_{i=k}^\infty \binom{i}{k}\lm^{i-k} p_i^{\eta}(t-s,x,z) |q_0(z)| \sum_{j=0}^\infty \lm^j p_j^{\eta}(s,z,y) \,dz\, ds\\
& = \sum_{i=k}^\infty \sum_{j=0}^\infty \binom{i}{k} \lm^{i+j-k}  p_{i+j+1}^{\eta}(t,x,y) \\
& = \sum_{n=k+1}^\infty  \sum_{j=0}^{n-k-1} \binom{n-j-1}{k} \lm^{n-k-1}  p_n^{\eta}(t,x,y)\\
& = \sum_{n=k+1}^\infty  \binom{n}{k+1} \lm^{n-k-1}
p_n^{\eta}(t,x,y),
\end{align*} where in the last equality we used the fact proved in the proof of \cite[Lemma 6]{2008-KB-WH-TJ-sm} (cf. \cite[(5.26)]{GKP}). The proof is complete.
\end{proof}

Next, we consider some properties of the function $\lambda\mapsto p^\lm(t,x,y)$.

\begin{lemma} For fixed $x,y\in \RR_0^d$ and $t>0$, the function
$$
h(\lm) = p^\lm(t,x,y), \quad \lm >0
$$
is completely monotone, i.e., $(-1)^kh^{(k)}(\lm)\ge0$ for all $\lm>0$ and $k=0,1,2,\cdots$.
\end{lemma}
\begin{proof}
For $\lambda> 0$, we take $\eta>\lambda$. Choosing $k=0$ in (\ref{eq:pkv_rel}), we get
$$
h(\lm) = \sum_{n=0}^\infty  (\eta-\lm)^n p_n^{\eta}(t,x,y)\,.
$$

By (\ref{eq:pkv_rel}), we get
\begin{align}\label{eq:deriv}
\frac{d^k }{d \lm^k} h(\lm)= {k!} \sum_{n=k}^\infty (-1)^k \binom{n}{k} (\eta-\lm)^{n-k} p^{\eta}_n(t,x,y) = (-1)^k \,{k!} \, p_k^{\lm}(t,x,y).
\end{align}
Since $p_k^{\lm}(t,x,y) \ge 0$, we conclude that $h$ is completely monotone on $(0,\infty)$.
\end{proof}

By the Bernstein theorem (see \cite[Theorem 1.4]{Schilling}), we get
\begin{corollary}\label{cor:meas_exists}
For fixed $x,y \in \RR_0^d$ and $t>0$, there exists a nonnegative Borel measure $\mu_{t,x,y}(du)$ on $[0,\infty)$ such that
$$
p^{\lm}(t,x,y) = \int_0^\infty e^{-\lm u}\, \mu_{t,x,y}(du).
$$
\end{corollary}
The next lemma will yield the monotonicity of the function $\lm \mapsto \frac{p_1^{\lm}(t,x,y)}{p^{\lm}(t,x,y)}$.

\begin{lemma}\label{lem:5}
For every $n \ge 1$, $\lambda\ge0$, $t>0$ and $x,y \in \RR_0^d$, we have
$$
(n+1)p_{n+1}^\lambda(t,x,y) p_{n-1}^\lambda(t,x,y) \ge n p_n^\lambda(t,x,y)^2.
$$
\end{lemma}
\begin{proof}
Fix $x,y\in \RR_0^d$ and $t>0$, and let $\mu = \mu_{t,x,y}$ be the nonnegative measure from Corollary \ref{cor:meas_exists}. Then, by (\ref{eq:deriv}),
$$
n!\, p^\lambda_n(t,x,y) = (-1)^n\,\frac{d^n}{d\lm^n} p^{\lm}(t,x,y) = \int_0^\infty e^{-\lambda u} u^n\, \mu(du)\,.
$$
According to the Cauchy-Schwarz inequality,
\begin{align*}
\left[\int_0^\infty e^{-\lambda u} u^n\, \mu(du)\right]^2 &= \left[\int_0^\infty e^{-\lambda u/2} u^{(n+1)/2} \cdot e^{-\lambda u/2}u^{(n-1)/2} \,\mu(du)\right]^2 \\
&\le \left(\int_0^\infty  e^{-\lambda u} u^{n+1} \,\mu(du)\right)\left(\int_0^\infty  e^{-\lambda u} u^{n-1} \,\mu(du)\right).
\end{align*}
Hence,
$$
(n! \,p_n^\lambda(t,x,y) )^2 \le (n+1)! \,p_{n+1}^\lambda(t,x,y) \cdot(n-1)! \,p_{n-1}^\lambda(t,x,y),
$$
and so the desired assertion follows.
\end{proof}

\begin{lemma}\label{lem:decr}
For fixed $x,y \in \RR_0^d$ and $t>0$, the function $\lm \mapsto \dfrac{p_1^{\lm}(t,x,y)}{p^{\lm}(t,x,y)}$ is decreasing on $(0,\infty)$.
\end{lemma}
\begin{proof} Let $H(\lambda)=\frac{-h'(\lambda)}{h(\lambda)}=\dfrac{p_1^{\lm}(t,x,y)}{p^{\lm}(t,x,y)}$ for $\lambda>0$, where in the second equality we used \eqref{eq:deriv}. Combining \eqref{eq:deriv} again with Lemma \ref{lem:5}, we find that
$$H'(\lambda)=\frac{-h''(\lambda)h(\lambda)+h'(\lambda)^2}{h^2(\lambda)}=\frac{-2p^\lambda_2(t,x,y)p^\lambda(t,x,y)+p_1^\lambda(t,x,y)^2}{p^\lambda(t,x,y)^2}\le 0,$$ which yields the desired assertion.
\end{proof}

We now present the main result in this appendix, which immediately gives us  Lemma \ref{lem:tpexp}.

\begin{theorem}\label{thm:wlb}
For every $\lm>0$, $t >0$ and $x,y \in \RR_0^d$, we have
$$
p^{\lm} (t,x,y) \ge p(t,x,y) \exp \left[\dfrac{\lm p_1(t,x,y)}{p(t,x,y)}\right]
$$
\end{theorem}
\begin{proof}
Fix $x,y \in \RR^d_0$ and $t>0$, and let $h(\lm) = p^{\lm}(t,x,y)$. Since $h \ge 0$,
$$
h(\lm) = h(0) \exp\left[\int_0^\lm (\ln h(u))' \,du \right] =  h(0) \exp\left[\int_0^\lm \dfrac{h'(u)}{h(u)} \,du \right].
$$
By \eqref{eq:deriv} and Lemma \ref{lem:decr}, we get
$$
p^{\lm}(t,x,y) = p(t,x,y) \exp\left[-\int_0^\lm \dfrac{p_1^{u}(t,x,y)}{p^{u}(t,x,y)}\, du\right] \ge p(t,x,y) \exp \left[-\dfrac{\lm p_1(t,x,y)}{p(t,x,y)}\right].
$$ The proof is complete.
\end{proof}

 \medskip

\noindent \textbf{Acknowledgements.} We  would like to thank Krzysztof Bogdan and Kamil Kaleta for interesting discussions and helpful comments.

\end{document}